\documentclass[preprint,a4j,leqno]{elsarticle}
\usepackage{amsmath,amssymb}
\usepackage{amsthm}

\usepackage{graphicx,color}
\usepackage[top=30truemm,bottom=30truemm,left=30truemm,right=30truemm]{geometry}

\usepackage{color}
\graphicspath{{./pic/}}

\journal{??}

\usepackage{bm}

\usepackage{makeidx}
\makeindex

\makeatletter

\newcommand{\Proofname}{Proof}
\makeatother

\def\theequation{\arabic{section}.\arabic{equation}}

\newtheorem{theorem}{Theorem}[section]
\newtheorem{lemma}[theorem]{Lemma}

\newtheorem{proposition}[theorem]{Proposition}
\newtheorem{Definition}[theorem]{Definition}
\newtheorem{Rem}[theorem]{Remark}
\newenvironment{definition}{\begin{Definition}\rm}{\end{Definition}}
\newenvironment{remark}{\begin{Rem}\rm}{\end{Rem}}

\newcommand\be{\begin{equation}}
\newcommand\ee{\end{equation}}
\newcommand\bea{\begin{eqnarray}}
\newcommand\eea{\end{eqnarray}}
\newcommand\beaa{\begin{eqnarray*}}
\newcommand\eeaa{\end{eqnarray*}}
\newcommand\bay{\begin{array}}
\newcommand\eay{\end{array}}
\newcommand\ba{\begin{align}}
\newcommand\ea{\end{align}}
\newcommand\beba{\begin{equation}\left\{\begin{array}{rcl}}
\newcommand\eeba{\end{array}\right.\end{equation}}
\newcommand\bebaa{\begin{equation*}\left\{\begin{array}{rcl}}
\newcommand\eebaa{\end{array}\right.\end{equation*}}
\newcommand\beca{\begin{equation}\left\{\begin{array}{rcll}}
\newcommand\eeca{\end{array}\right.\end{equation}}
\newcommand\becaa{\begin{equation*}\left\{\begin{array}{rcll}}
\newcommand\eecaa{\end{array}\right.\end{equation*}}

\newcommand{\ph}{\varphi}

\newcommand{\R}{\mathbb{R}}
\newcommand{\Z}{\mathbb{Z}}
\newcommand{\N}{\mathbb{N}}

\newcommand{\C}{\mathbb{C}}

\newcommand{\ds}{\displaystyle}

\renewcommand{\theequation}{\arabic{section}.\arabic{equation}}

\graphicspath{{./pic/}}

\begin{document}
\begin{frontmatter}
\title{Method of  fundamental solutions  for  Neumann problems of the modified Helmholtz equation in disk domains}

\author[Hokudai]{Shin-Ichiro Ei}
\author[Kyushu]{Hiroyuki Ochiai}
\author[Hakodate]{Yoshitaro Tanaka }

\address[Hokudai]{ 
Department of Mathematics, 
Faculty of Science, 
Hokkaido University,
Sapporo, 060-0810, Japan}

\address[Kyushu]{
Institute of Mathematics for Industry, Kyushu University, Fukuoka, 819-0395, Japan
}

\address[Hakodate]{
Department of Complex and Intelligent Systems,
School of Systems Information Science,
Future University Hakodate,
Hakodate, 041-8655, Japan
}

\begin{abstract}
The method of the fundamental solutions (MFS) is used to construct an approximate solution for a partial differential equation in a bounded domain.
It is demonstrated by combining the fundamental solutions shifted to the points outside the domain and determining the coefficients of the linear sum to satisfy the boundary condition on the finite points of the boundary.
In this paper, the existence of the approximate solution by the MFS for the Neumann problems of the modified Helmholtz equation in disk domains is rigorously demonstrated.
We reveal the sufficient condition of the existence of the approximate solution.
Applying Green's theorem to the Neumann problem of the modified Helmholtz equation, we bound the error between the approximate solution and exact solution into the difference of the function of the boundary condition and the normal derivative of the approximate solution by boundary integrations.
Using this estimate of the error, we show the convergence of the approximate solution by the MFS to the exact solution with exponential order, that is, $N^2a^N$ order, where $a$ is a positive constant less than one and $N$ is the number of collocation points.
Furthermore, it is demonstrated that the error tends to $0$ in exponential order in the numerical simulations with increasing number of collocation points $N$.

\end{abstract}


\end{frontmatter}

\section{Introduction}\label{sec:intro}
The method of fundamental solutions (MFS), also known as the charge simulation method, is used to construct an approximate solution for a partial differential equation in a bounded domain.
It is based on combining the shifted fundamental solutions of the partial differential equation
and controlling the coefficients of the linear sum of the fundamental solutions to satisfy the boundary condition on the finite points of the boundary.
More precisely, the approximate solution by the MFS is constructed by the following procedure.
The points outside the domain, called the charge points, are first introduced and, thereafter, the fundamental solutions for the partial differential equation are shifted to each charge point.
This renders the singular point of each fundamental solution laid outside the domain.
If the differential operator of the equation is invariant to the shift, the approximate solution is given by the linear sum of the shifted fundamental solutions.
Next, a number of points on the boundary of the domain, called the collocation points, greater than or equal to the number of charge points are prepared.
The coefficients of the linear sum of the fundamental solutions are determined to satisfy the boundary condition on the finite collocation points.
As the number of collocation points increases,
the approximate solution possesses more points to satisfy the boundary condition.
As a result, we can construct the approximate solution, which is expected to converge to the exact solution for the partial differential equation if the number of collocation points increases.

The arrangement of both charge and collocation points generally depends on the shape of the domains.
The conditions of these arrangements have been studied mathematically.
There are several ways to determine the coefficients of the linear sum of the fundamental solutions, for example, the collocation method \cite{kats-okamo1988,katsurada1990, Li2004,Sakaki2017}, the collocation Trefftz method \cite{Z.Li2017}, and the hyperinterpolation method \cite{Li2008}.
From the analytical studies, it has also been reported that the approximate solutions by the MFS for some partial differential equations converge to the exact solutions with exponential order with respect to the number of the collocation points \cite{kats-okamo1988,katsurada1990,Sakaki2017}.
Most results for the convergence of the approximate solution in the MFS are for the Dirichlet boundary problems.
The convergence results of the approximate solution by the MFS for the Dirichlet boundary problem of the Laplace equation have been shown in the disk domain, Jordan regions with analytic boundary, and the doubly connected regions in \cite{kats-okamo1988,katsurada1990,Sakaki2017}, respectively.
The convergence results for the Dirichlet boundary problem of the Helmholtz equation in the disk domain and other analytic domains have been reported
in \cite{Barnett2008}.
The convergence results for the Dirichlet boundary problem of the modified Helmholtz equation in the disk domain and sphere domain have been reported in \cite{Li2004,Li2008}.
For the Neumann boundary problem of the Laplace equation, the convergence result of the approximate solution by the MFS has been reported in the simply connected domain except for disk domains \cite{Z.Li2017}.
According to these earlier studies of the MFS, the approximate solutions are often constructed in the disk domain, and thereafter the results are extended to the general shape of the domain.

As an application of the Neumann boundary problem of the modified Helmholtz equation,
its solution is used to investigate the interactions between the domain shape and the motion of the inside traveling spot and standing pulse solution of a reaction diffusion system \cite{Ei}.
The traveling spots and the standing pulse are the solutions that allow spatially localized patterns to propagate.
They are often observed in reaction diffusion systems.
The influence of the shape of the domain on the motion of the inside traveling spots and the standing pulse is analyzed by deriving the equation of motion \cite{Ei}.
However, this theory requires the explicit form of the solution of the modified Helmholtz equation to be obtained with the Neumann boundary condition for the various domain shapes.
Although the existence of the solution of this modified Helmholtz equation is simply shown, the explicit form of the solution has not been obtained except for domains with certain shapes.

In light of this background, we construct the approximate solution for the Neumann problem of the modified Helmholtz equation, particularly in the disk domain, by using the MFS with a collocation method as a quick step.
We employ the collocation method to construct the approximate solution in the MFS.
Estimating the ratio of the modified Bessel functions of the first and second kinds,
we reveal the sufficient condition that all eigenvalues of the matrix generated by the MFS are not equal to zero.
This gives us the existence of the unique approximate solution by the MFS for this problem.
Moreover, we provide an algorithm to calculate the error bound between this approximate solution in the MFS and the exact solution found by the energy method.
Using this error bound, we show \textit{a priori} that the approximate solution converges to the exact solution in exponential order, that is, $N^2 a^N$ with $0<a<1$ and the number of collocation points $N$.
Furthermore, through numerical simulation, we demonstrate \textit{a posteriori} that the error tends to $0$ in exponential order as the number of collocation points $N$ increases.

This paper is organized as follows:
In Section \ref{sec:2}, we state the mathematical setting of the Neumann boundary problem of the modified Helmholtz equation, the MFS with the collocation method, and the main results.
We prove the main results, Theorems \ref{thm:1} and \ref{thm:2}, in Section \ref{sec:proof1}, and Sections \ref{sec:proof2} and \ref{sec:conv}, respectively.
In Section \ref{sec:NV}, we describe the algorithm used to calculate the error bound for this approximate solution in the MFS.
In Section \ref{sec:simu}, we present the result of the numerical simulations, and we summarize our paper in Section \ref{sec:summary}.

\section{Mathematical settings and main results}\label{sec:2}
As explained in Section \ref{sec:intro}, we consider the following form of the Neumann boundary problem of the modified Helmholtz equation:
\begin{equation}
\label{eq:gene_prob}
\left \{
\begin{aligned}
&\Delta g -\alpha^2 g=0, \quad x \in \Omega, \\
&\dfrac{\partial g}{\partial n}= s, \quad x \in \partial\Omega,
\end{aligned}
\right.
\end{equation}
where $\Omega := B (0, R) \subseteq \R^2, \ B(0,R)$ is a disk with a radius $R>0$ at the origin, $\alpha$ is a positive constant, and $s=s(x)$ is a $2 \pi$-periodic and sufficiently smooth function.
Here we set the notation as $	S(\theta):= s(Re^{i \theta}) $ for $\theta \in [0, 2\pi)$.
A typical example of this equation applied to the pulse motion in the reaction diffusion systems as introduced in Section \ref{sec:intro} is as follows \cite{Ei}:
\begin{equation}
\label{eq:ei_prob}
\left \{
\begin{aligned}
&\Delta g -\alpha^2 g=0, \quad x \in \Omega, \\
&\dfrac{\partial g}{\partial n}= - \dfrac{\partial \Phi(|x-P|)}{\partial n}, \quad x \in \partial\Omega,
\end{aligned}
\right.
\end{equation}
where $P$ is the fixed point in $\Omega$, and $\Phi=\Phi(r), \ r = |x| = (x_1^2 + x_2^2 )^{1/2}$ is a sufficiently smooth function satisfying $\Phi(r) \to e^{-\alpha r}/ \sqrt{r}$ as $r \to \infty$.
The point $P$ corresponds to the position of a pulse as in Fig. \ref{fig:MFS} (a), and
the derivation of this equation can be found in \cite{Ei}.
\begin{figure*}[bt]
\begin{center}
\begin{tabular}{cc}
\includegraphics[width=6cm, bb=0 0 214 189]{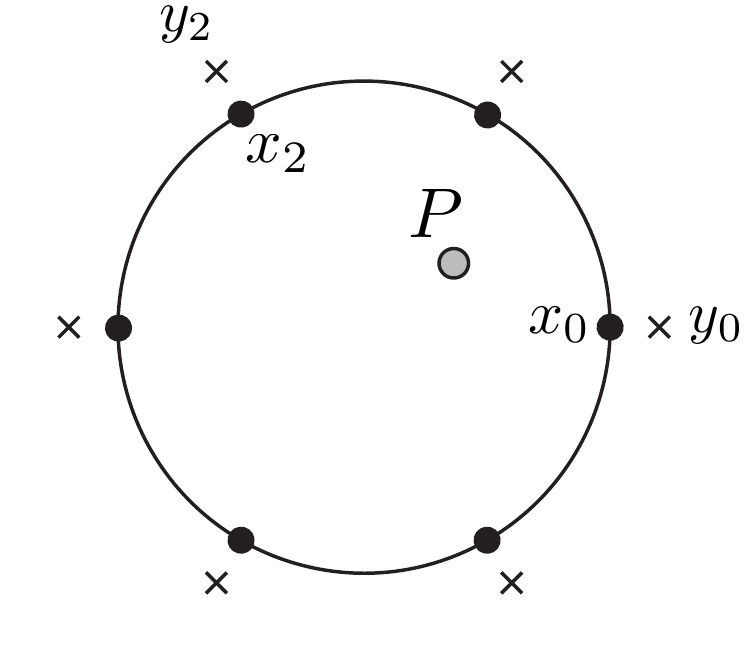}
&\includegraphics[width=6cm, bb=0 0 358 311]{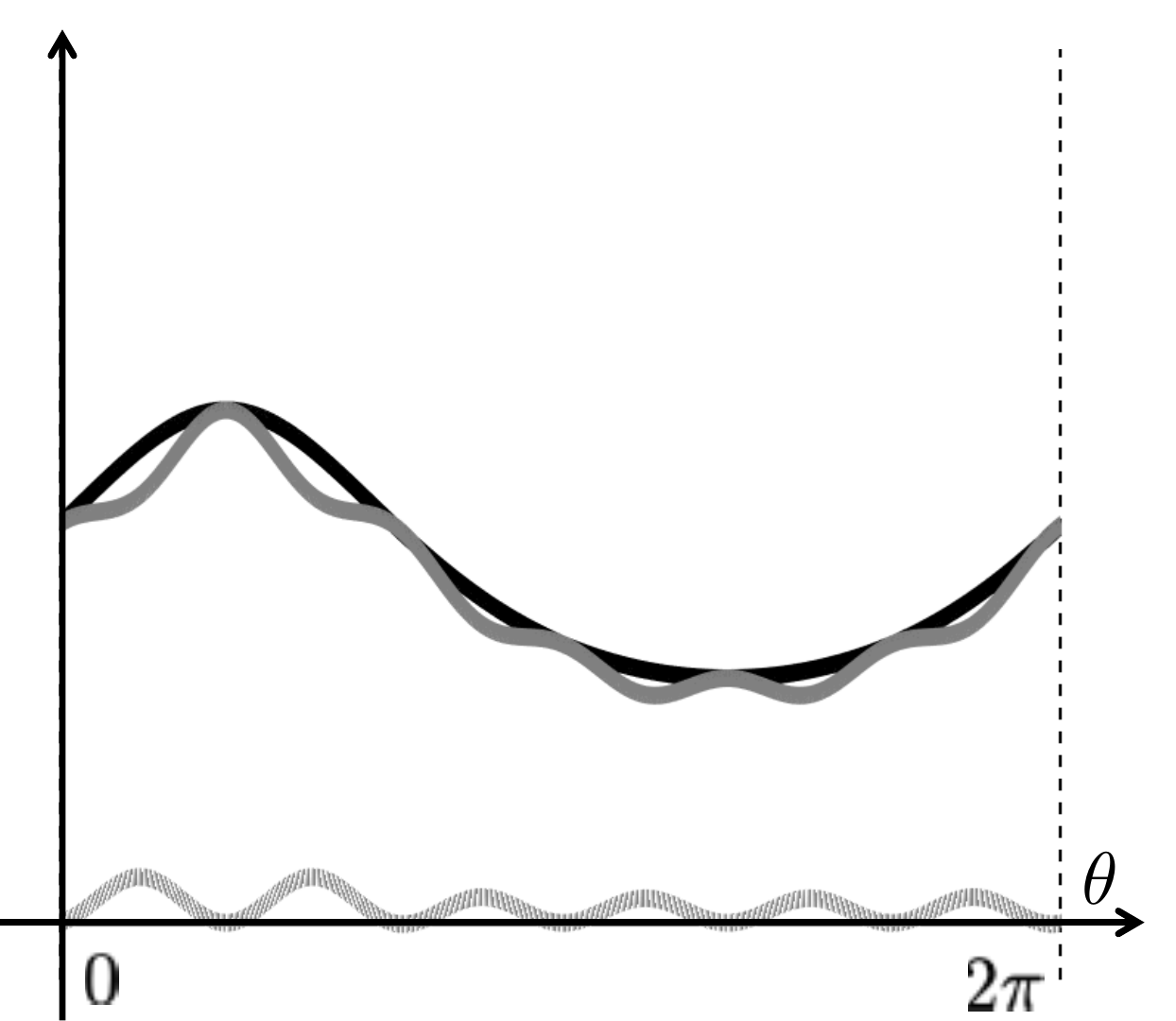}\\
(a) &(b)
\end{tabular}
\end{center}
\caption{(a) Schematic of the domain $\Omega$, position of the pulse $P$, the charge points $\{y_k\}_{k=1}^N$, and the collocation points $ \{x_j\}_{j=1}^N $ in the collocation method.
(b) The profiles of the normal derivative of the exact solution $\partial g/\partial n= s $, the approximate solution $\partial g_N/\partial n$, and the error $ \partial h_N/ \partial n$ on $\partial \Omega$ in \eqref{eq:ei_prob} with the parameters $\rho= 3$, $R=1$, $P=0.2e^{\pi i/3}$, $N=6$, and $\Phi(r) = e^{-\alpha r}/ \sqrt{r}$ with $\alpha=1.0$ and $r=|x|$, where $h_N$ is defined in Section \ref{sec:NV}.
The black, gray, and dashed curves correspond to the profiles of $\partial g/\partial n $, $\partial g_N/\partial n$, and $ \partial h_N/ \partial n$, respectively.
}
\label{fig:MFS}
\end{figure*}

The existence of the solution of \eqref{eq:gene_prob} is shown by the standard arguments.
In this argument $\Omega $ may be replaced by the general bounded domain with smooth boundary.
Supposing that the function $ s$ is the normal derivative of the trace of some function, say $\sigma$, and it satisfies $\partial \sigma / \partial n = s$, we change the variable as $u = g + \sigma$ in \eqref{eq:gene_prob}.
Then we have
\begin{equation}
\label{eq:ei_prob2}
\left \{
\begin{aligned}
&\Delta u -\alpha^2 u = \Delta \sigma -\alpha^2 \sigma =: \psi, \quad x \in \Omega, \\
&\dfrac{\partial u}{\partial n}= 0, \quad x \in \partial\Omega.
\end{aligned}
\right.
\end{equation}
We define the weak solution of \eqref{eq:ei_prob2} as follows.
\begin{definition}
We say that a function $u \in H^1(\Omega) $ is a weak solution of \eqref{eq:ei_prob2} if
\[
\int_{\Omega} \nabla u \cdot \nabla v dx + \alpha^2 \int_{\Omega} u v dx = - \int_{\Omega} \psi v dx
\]
for all $v \in H^1(\Omega)$.
\end{definition}
By using the Lax--Milgram theorem we show the existence of a unique weak solution $u \in H^1(\Omega)$ of \eqref{eq:ei_prob2}. 
The regularity theory refines that the weak solution $u \in H^1(\Omega)$ belongs to $C^{\infty}(\Omega)$ if $\psi \in C^{\infty}(\Omega)$.
Summarizing the previous results, we have the following proposition.
\begin{proposition}
Assume that $\sigma \in C^{\infty}(\Omega)$.
Then, there exists a unique solution $ g \in C^{\infty}(\Omega) $ of \eqref{eq:gene_prob}.
\end{proposition}
In particular, in the case of $\Omega = B(0, R)$ we have the exact solution of \eqref{eq:gene_prob} as follows.
\begin{proposition}\label{prop:exact}
Let the function in $(r, \theta) \in [0, \infty)\times [0,2\pi) $ be
\begin{align}
g(re^{i\theta})	
	& = \sum_{n \in \Z} \frac{ I_n (\alpha r) } { \alpha I'_n(\alpha R) } a_n e^{i n \theta},\label{eq:exact_g}\\
a_n & := \frac{1}{2 \pi} \int_{-\pi}^{\pi} S(\theta)e^{-i n \theta} d\theta, \notag
\end{align}
where $I_n(x)$ is the modified Bessel function of the first kind of $n$ order, and $ a_n $ is the Fourier coefficients for $S(\theta)$, and it satisfies
\begin{equation*}
	S(\theta)= \sum_{n \in \Z} a_n e^{i n \theta}.
\end{equation*}
Then $g(r e^{i \theta})$ satisfies \eqref{eq:gene_prob}.
\end{proposition}
The proof is presented in Appendix \ref{append:exact}.
This exact solution is derived by the typical variable separation of the Fourier method.

Next, we construct the approximate solution for \eqref{eq:gene_prob} by the method of fundamental solution.
Owing to the operator $\Delta - \alpha^2$, the fundamental solution of \eqref{eq:gene_prob} is written in terms of the modified Bessel function of the second kind, $K_0( \alpha |x|)$.
We take the $N$ points $\{ y_k\}_{k=1} ^{N}$ outside of $\Omega$, called the charge points.
As the operator $\Delta - \alpha^2$ is invariant to the shift, the fundamental solutions $K_0(\alpha |x - y_k|), \ (k=1, \ldots, N)$, of which singular points are shifted to the outside of $\Omega$, satisfy the principal equation \eqref{eq:gene_prob}.
Therefore, we show that
\begin{equation}\label{sol:mfs}
g_N(x):= \sum_{k=1}^N Q_k K_0 ( \alpha | x - y_k |)	
\end{equation}
with constants $\{ Q_k\}_{k=1}^N $ satisfies the principal equation \eqref{eq:gene_prob} in $\Omega$.
From here, we identify the two-dimensional Euclid space $\R^2$ with the complex plane $\C$.
We introduce the $N$ points $\{ x_j\}_{j=1}^N$ on the boundary $\partial \Omega$, called the collocation points, so that the approximate solution $g_N(x)$ satisfies the boundary condition on the $N$ points of the boundary.
Substituting $g_N(x)$ for the boundary condition of \eqref{eq:gene_prob} and employing $\{ x_j\}_{j=1}^N$ on the $\partial \Omega$, we have
\begin{equation}\label{eq:renri}
	\frac{\partial}{\partial n}g_N(x_j )=\sum_{k=1}^{N}Q_k \frac{\partial}{\partial n}K_0( \alpha | x_j - y_k|) = s(x_j),
	\quad (j=1, \ldots, N).
\end{equation}
We write \eqref{eq:renri} as a system of linear equations
for unknown functions $\{Q_k\}_{k=1}^{N}$ in a matrix form.
By using $K'_0(x)= - K_1(x)$, we set
\begin{equation*}
\begin{split}
c_{k,j}&:=- \alpha K_1(\alpha |x_j-y_k|)\frac{|x_j| - \operatorname{Re} \Big( ( x_j / |x_j| ) \cdot {\bar y_k} \Big)}{|x_j-y_k|},\\
s_{j}&:= s(x_j),
\end{split}
\end{equation*}
for $k,j=1, \ldots, N$, where $\operatorname{Re}$ means the real part of complex number, and ${\bar y_k}$ means the complex conjugate of $y_k$.
The form of $s_j$ in \eqref{eq:ei_prob} is given by
\[
s_{j}:=-\Phi'( |x_j - P|)\frac{|x_j| - \operatorname{Re} \Big( ( x_j  / |x_j| ) \cdot {\bar P} \Big)}{|x_j-P|}, \quad (j=1, \ldots, N).
\]
This vector is used in the numerical simulation as in Figs. \ref{fig:MFS} and \ref{fig:error}.
We define $G=\{ c_{k,j}\}_{1\le k,j \le N}$, $Q=(Q_1,\cdots, Q_N)^{T}$ and $S=(s_1,\cdots,s_N)^{T}$.
Then \eqref{eq:renri} is equivalent to the following system of linear equation:
\begin{equation}\label{eq:gs}
GQ = S.	
\end{equation}
For a solution $\{Q_k\}_{k=1}^{N}$ of this system \eqref{eq:gs},
we obtain an approximate solution $g_N(x)$ given by \eqref{sol:mfs}.

Next, we explain the collocation method.
From here, $\Omega$ is denoted as a disk again, that is, $\Omega = B(0,R)$ with a positive constant $R>0$.
In the typical collocation method in a disk domain \cite{kats-okamo1988}, the charge points and the collocation points are usually set as $\{y_k\}_{k=1}^N =\{ \rho \omega^k \}_{k=1}^N$, and as $\{x_j\}_{j=1}^N=\{ R \omega^j \}_{j=1}^N$ on $\partial \Omega$ in $\C$, respectively,
where $0<R<\rho$, and $ \omega = e^{2\pi i/N}$.
These points divide the circumference of the concentric circles into $N$ equal parts.
The schematics are presented in Fig. \ref{fig:MFS} (a).
Then, calculating that
\[
|x_j - y_k | = |R - \rho \omega^{k-j}|,
\]
we have
\begin{equation*}
	c_{k,j}
= - \alpha K_1(\alpha |R-\rho\omega^{k-j}|)\frac{R - \operatorname{Re}( \rho \omega^{-k+j})}{|R-\rho \omega^{k-j}|}.
\end{equation*}
Therefore, using the notation $l:=-(k-j)$, we see that $c_{k,j}=c_l$, and that
\[
c_l : = - \alpha K_1(\alpha |R-\rho e^{-i\theta_l }|)\frac{R - \rho \cos\theta_l }{|R-\rho e^{-i\theta_l }|}
, \quad \theta_l := \frac{2\pi l}{N}.
\]
Then the matrix $G$ becomes a cyclic matrix described by
\begin{equation*}
		G=
	\begin{pmatrix}
	c_0 & c_1 & \cdots & \cdots &c_{N-1} \\
	c_{N-1} & c_0 & \ddots & &\vdots \\
	\vdots & &\ddots & \ddots& \vdots \\
	\vdots & & & \ddots & c_1 \\
	c_1 &\cdots & \cdots & \cdots & c_0
	\end{pmatrix}.
\end{equation*}
The eigenvalues of this matrix are calculated by $\sum_{ l = 0 }^{N-1}c_l \omega^{ l m }, \ m= 0, \ldots, N-1$, from the discrete Fourier transform.

By calculating the inverse matrix of $G$ we explicitly obtain the coefficients $\{ Q_{k} \}_{k=1}^N$.
We denote the inverse matrix of $G$ by $G^{-1}$.
Using the Lagrangian interpolation polynomial as in Appendix \ref{append:Lag}, we obtain $G^{-1}$.
Setting
\begin{equation*}
	f(x) := \sum _{l=0}^{N-1} c_l x^l,
\end{equation*}
we write the eigenvalues as
\begin{equation}\label{f_omega1}
	f(\omega^m) = \sum _{l=0}^{N-1} c_l \omega^{ml}, \quad (m=0,\ldots, N-1).
\end{equation}
Furthermore, as $G^{-1}$ is also a cyclic matrix, we define the component of $G^{-1}$ as $b_l, \ (l=0, \ldots, N-1)$, by labeling the index similarly to that of $G$.
Then, we have
\begin{equation*}
	b_l\text{=} \frac{1}{N} \sum _{m=0}^{N-1} \frac{1}{f\left(\omega ^m \right) \omega ^{m l}}, \quad (l=0, \ldots, N-1)
\end{equation*}
 as computed in Appendix \ref{append:Lag}.
If $f(\omega^m) \ne 0$ for $m=0, \ldots, N-1$, which implies the matrix $G$ has no zero eigenvalue,
we calculate the coefficients $Q_k$ by
\begin{equation}\label{eq:qk}
	Q_k = \sum _{l=1}^N s_l b_{(-k+l)}.
\end{equation}
By using this $\{ Q_k \}_{k=1}^N$, the approximate solution $g_N(x)$ is given by \eqref{sol:mfs}.
Under these settings the main results are as follows.
\begin{theorem}
\label{thm:1}
Suppose that $\rho > \Big( \sqrt{ 4\alpha^2 R^2 +6 - 2 \sqrt{ 4\alpha^2 R^2+9 } } \Big)/\alpha $ for any $R>0$ and $\alpha>0$.
Then, we have $f(\omega^m) > 0$ for any $N \in \N$, and $m=0, \ldots, N-1$.
Thus, the approximate solution $g_N(x)$ is determined uniquely.
\end{theorem}
The proof of Theorem \ref{thm:1} is given in Section \ref{sec:proof1}.
Moreover, using Green's theorem as explained in Section \ref{sec:NV}, we can estimate the error bound between exact solution and the approximate solution in $H^2(\Omega)$ by boundary integrations.
Applying the Sobolev embedding theorem, we obtain the following convergence result.
\begin{theorem}\label{thm:2}
In addition to the hypothesis of Theorem \ref{thm:1}, we assume that $g$ can be extended to the neighborhood of ${\bar \Omega}$ and, hence, that $g$ is bounded in $0 \le r \le r_0$, where some $r_0>R$.
Then, there exist constants $C_1>0$ and $a$ with $0<a<1$ that are independent of $N$ and $g$ such that
\begin{equation*}
	\sup_{ x \in \Omega }| g- g_N | \le C_1 N^2  a^N \sup_{|x|\le r_0}| g (x)|.
\end{equation*}
\end{theorem}
Hereafter $C_j, \ (j\in \N)$, denotes positive constants.
The proof of Theorem \ref{thm:2} is given in Section \ref{sec:conv} with the preparation in Sections \ref{sec:NV} and \ref{sec:proof2}.
\begin{remark}
In practice, assuming $\rho \ge 2 R $ is a simpler and more sufficient condition for Theorem \ref{thm:1}.
As $ \sqrt{ 4\alpha^2 R^2 +6 - 2 \sqrt{ 4\alpha^2 R^2+9 } }< 2 \alpha R $ for any $\alpha, R>0$, the condition of Theorem \ref{thm:1} becomes a little stronger by $\rho \ge 2 R$.
\end{remark}
\begin{remark}
From the variable separation of the Fourier method, we can construct the exact solution ${\tilde g}(r,\theta)$ for the same problem in $B(0,r_0)$ as \eqref{eq:gene_prob}, and obtain ${\tilde g}(R,\theta)$.
This implies that the solution $g$ of \eqref{eq:gene_prob} can be extended to $B(0,r_0)$ if $S( \theta ) = \partial {\tilde g} /\partial n (R,\theta)$.
Thus, the assumption of Theorem \ref{thm:2} is not empty at least.
\end{remark}

\section{Existence of the approximate solution}\label{sec:proof1}
In this section we give the proof of Theorem \ref{thm:1}.
We set
\begin{align*}
A_n
&:=( 1+n )K_{ n+1}(\alpha \rho)I_{ n+1}(\alpha R), \quad (n \in \N),
\end{align*}
where  $K_{n}(x)$ is the modified Bessel function of the second kind of $n$ order.
Furthermore, we introduce the following functions:
\begin{equation*}
\begin{split}
	d(\theta)&:=\left| R - \rho e^{- i \theta} \right|=\sqrt{R^2+\rho^2-2 R \rho \cos\theta},\\
	c(\theta)&:= - \alpha K_1(\alpha |R-\rho e^{-i\theta}|)\frac{R - \rho \cos\theta}{|R-\rho e^{-i\theta}|}
	= - \alpha K_1(\alpha d (\theta))\frac{R - \rho \cos\theta}{d (\theta) }
\end{split}
\end{equation*}
for $\theta \in [0, 2\pi)$.
To prove Theorem \ref{thm:1} we give the Fourier series expansion of $c(\theta)$, and  another description of the eigenvalues $f(\omega^m)$ for $m=0,\ldots, N-1$ by the Fourier coefficient of $c(\theta)$ in the following proposition.

\begin{proposition}\label{thm:3}
The Fourier series expansion of $c(\theta)$ is given by 
\begin{align}
c(\theta)  
&= \frac1R \sum_{n \in \Z} \tilde{A}_n e^{i n \theta}, \label{eq:c}
\end{align}
where
\begin{align*}
\tilde{A}_n 
&:= \sum_{r=0}^\infty (A_{|n-1|+2r} + A_{|n+1|+2r} -\frac{2R}{\rho} A_{|n|+2r}), \quad (n \in \Z ).
\end{align*}
Moreover, the eigenvalues of the cyclic matrix $G$ are expressed by
\begin{align}
 f(\omega^m)
&= \frac{N}{R}  \sum_{ n \in \Z} \tilde{A}_n \delta _ { n + m \in N\Z} 
= \frac{N}{R}  \sum_{ n \in \Z} \tilde{A}_{ nN + m }, \quad (m=0,\ldots, N-1),
\label{eq:f_omega}
\end{align}
 where $\delta_n$ is the Kronecker Delta. 

\end{proposition}

\begin{proof}[Proof of Proposition\,\ref{thm:3}]
We utilize the following formulas of the modified Bessel function of the second kind, and the Gegenbauer polynomial \cite[p365]{Watson}:
\begin{equation*}
\begin{split}
	\frac{K_{\nu}(d(\theta))}{d(\theta)^\nu}
	&=2^\nu\Gamma(\nu)\sum_{m=0}^\infty (\nu+m)\frac{K_{m+\nu}(\rho)}{\rho^\nu} \frac{I_{m+\nu}(R)}{R^\nu} C^\nu_m(\cos\theta), \quad (\rho>R),\\
	C^1_m(\cos\theta)&=\sum_{r=0}^m \cos(2r-m)\theta
\end{split}
\end{equation*}
for $\theta \in [0, 2\pi)$, where $\nu \in \N$, $\Gamma(\nu)$ is the Gamma function, $C^{\nu}_m(x)$ is the Gegenbauer polynomial of $m$ order.
Here we compute that 
\begin{align*}
	C^1_m(\cos\theta)&=\sum_{r=0}^m \cos(2r-m)\theta 
	= \sum_{r=0}^m e^{i(2r-m)\theta} 
	= \sum_{r=-m, r-m\in 2\Z}^m e^{ir\theta}.
\end{align*}
We deal with the specialization to $\nu=1$:
\begin{align*}
	\frac{\alpha K_{1}(\alpha d(\theta))}{d(\theta)}
	&= \frac{2}{R \rho}  \sum_{m=0}^\infty A_m C^1_m(\cos\theta).
\end{align*}
Putting the modified Bessel functions and the Gegenbauer polynomial into $c(\theta)$, we compute as
\begin{align*}
c(\theta) 
&= (\rho \cos \theta-R) \frac{\alpha K_1(\alpha d (\theta))}{d (\theta) }\\
	&= \frac{2}{R \rho} (\rho \cos \theta-R)
	 \sum_{m=0}^\infty A_m C^1_m(\cos\theta) \\
&= \frac{2}{R \rho} (\rho \cos \theta-R)
	 \sum_{m=0}^\infty A_m \sum_{ n = -m, n-m\in 2\Z}^m e^{ i n \theta}\\
&= \frac{2}{R \rho} (\rho \cos \theta-R) \sum_{ n \in \Z} \sum_{m =|n|, m-n \in 2\Z}^\infty A_m e^{ i n \theta}\\
&= \frac{2}{R \rho} (\rho \cos \theta-R) \sum_{ n \in \Z} \sum_{r=0}^\infty A_{|n|+2r} e^{i n \theta}\\
&= \frac{1}{R} (e^{i\theta}+e^{-i \theta}-\frac{2R}{\rho}) \sum_{ n \in \Z} \sum_{r=0}^\infty A_{ | n | + 2 r } e^{i n \theta}\\
&= \frac{1}{R}  \sum_{ n \in \Z} \sum_{r=0}^\infty A_{ |n| + 2r } ( e^{i (n+1) \theta} + e^{ i( n - 1 )\theta }
-\frac{2R}{\rho} e^{ i n \theta})\\
&= \frac{1}{R}  \sum_{ n \in \Z} \sum_{r=0}^\infty (A_{ | n - 1 |+2r} + A_{ | n + 1 | + 2 r } -\frac{2R}{\rho} A_{ | n | + 2 r }) e^{ i n \theta}.
\end{align*}
So we put 
\begin{align*}
\tilde{A}_n = \sum_{r=0}^\infty (A_{|n-1|+2r} + A_{|n+1|+2r} -\frac{2R}{\rho} A_{|n|+2r}).
\end{align*}
Then the Fourier series expansion of $c(\theta)$ is given by
\begin{align*}
c(\theta)  = \frac1R \sum_{ n \in \Z} \tilde{A}_n e^{i n \theta}.
\end{align*}
We see that $\tilde{A}_{-n}=\tilde{A} _ n $ for $ n \in \Z$.
Then we obtain that
\begin{align}\label{tA:natural}
\tilde{A}_n = 
\sum_{r=0}^\infty (A_{ n-1+2r } + A_{ n+1+2r } -\frac{ 2R } { \rho } A_{ n + 2r } )
\end{align}
if $ n \in \N $, and that
\begin{align*}
\tilde{A}_0 
= \sum_{r=0}^\infty (2A_{1+2r}  -\frac{2R}{\rho} A_{2r}).
\end{align*}
Substituting \eqref{eq:c} into \eqref{f_omega1}, we have
\begin{align*}
 f(\omega^m)
&= \sum_{j=0}^{N-1} c(\theta_j) \omega^{mj}  \notag\\
&=\sum_{j=0}^{N-1} \frac{1}{R}  \sum_{ n \in \Z } \tilde{A} _ n \omega^{ n j } \omega^{mj}  \notag\\
&= 
\frac{1}{R}  \sum_{ n \in \Z} \tilde{A}_n
\sum_{j=0}^{N-1} \omega^{ ( n +m)j } \notag\\
&= \frac{N}{R}  \sum_{ n \in \Z} \tilde{A}_n
\delta_{ n+m \in N\Z}\\
&= \frac{N}{R}  \sum_{ n\in \Z} \tilde{A}_{ nN+m}.
\end{align*}
\end{proof}

From the expression \eqref{eq:f_omega}, we show the positivity of $f(\omega^m)$ with $  m=0, \ldots, N-1$.
As the preparations of the proof of Theorem \ref{thm:1}, we will show the following lemmas.
\begin{lemma}\label{lemm:An-1}
Assume that $x > \sqrt{  4y^2+6 - 2 \sqrt{ 4y^2+9 } }$ and $0<y<x$ 
 for  $ x, y\in \R$. 
 Then 
\begin{align*}
\sup_{ n \in \N } \frac {y}{x} \frac{ 1+n }{n} \frac{ I_{n+1}( y ) }{ I_{n}( y )} \frac{ K_{n+1}( x ) }{ K_{n}( x )} < 1.
\end{align*}
Therefore, replacing $x=\alpha \rho$ and $y=\alpha R$, we have  
\begin{equation} \label{an/an-1}
 \sup_{n \in \N} \frac{ R }{\rho }  \frac{  A_n } {  A_{n-1} } < 1.
\end{equation}
\end{lemma}
\begin{proof}[Proof of Lemma\,\ref{lemm:An-1}]
We use the following inequalities:
By referring to (\cite[Theorem \ 1.2]{Laforgia}), we have
\begin{equation*}
	\frac{ K_{\nu} (x) }{ K_{\nu-1} (x) } < \frac{ \nu + \sqrt { \nu^2 + x^2 } }{ x }, \quad (x>0, \ \nu \in \R),
\end{equation*}
and from (\cite[Theorem \ 3]{segura}), we see that
\begin{equation}\label{ineq:I_n}
0< \frac{ I_{\nu + \frac{1}{2}} (y)}  {I_{ \nu - \frac{1}{2} }(y) } < \frac{ y }{ \nu + \sqrt{ \nu^2 +y^2 } }, \quad (y>0, \ \nu \ge 0).
\end{equation}
Using \eqref{ineq:I_n}, we have for $y>0$
\begin{align*}
\frac{ I_{n + 1}  (y)}  {I_{ n }(y) } < \frac{ y }{ (n+ 1 / 2) + \sqrt{ (n+1/2)^2 +y^2 } }.
\end{align*}
Thus, we calculate that 
\begin{align}
& \frac {y}{x} \frac{ 1+n }{n} \frac{ I_{n+1}( y ) }{ I_{n}( y )} \frac{ K_{n+1}( x ) }{ K_{n}( x )} \notag\\
&<  \frac {y}{x} \frac{n+1}{n}      \frac{ y }{ (n+ 1 / 2) + \sqrt{ (n+1/2)^2 +y^2 } }     \frac{ n+1 + \sqrt { (n+1)^2 + x^2 } }{ x } \notag\\
&=  \frac {y^2}{x^2} \frac{ n+1 }{n}    \frac{ n+1 + \sqrt { (n+1)^2 + x^2 } }{  (n+ 1 / 2) + \sqrt{ (n+1/2)^2 +y^2 }  }.\label{ineq:m_i_s}
\end{align}
Regarding the sequence $  ( n+1 + \sqrt { (n+1)^2 + x^2 } ) / ( (n+ 1 / 2) + \sqrt{ (n+1/2)^2 +y^2 } ) $ as the function of $n \in \R$ and calculating the derivative, we see that the sequences,  and  $(n+1)/n$  are monotonically decreasing with respect to $n$. 
Thus, \eqref{ineq:m_i_s} attains the maximum value when $n=1$,
thereby estimating that 
\begin{align*}
 \frac {y}{x} \frac{ 1+n }{n} \frac{ I_{n+1}( y ) }{ I_{n}( y )} \frac{ K_{n+1}( x ) }{ K_{n}( x )} 
<  \frac {4y^2}{x^2} \frac{  2 + \sqrt { 4 + x^2 }  }{  3 + \sqrt { 9 + 4 y^2 }  }
= \frac{  \sqrt{ 4y^2+9 } - 3  }{ \sqrt{  x^2+4 } - 2 }.
\end{align*}
Solving the  inequality with respect to $x$ so that $   (  \sqrt{ 4y^2+9 } - 3  ) /  ( \sqrt{  x^2+4 } - 2  ) < 1$, we have 
\[
x > \sqrt{  4y^2+6 - 2 \sqrt{ 4y^2+9 } }.
\]
Therefore the statement of this lemma is shown, and by substituting $x=\alpha \rho$ and $y=\alpha R$, we have \eqref{an/an-1}
when $  \rho  >  \Big( \sqrt{  4\alpha^2 R^2 + 6 - 2 \sqrt{ 4\alpha^2 R^2+9 } } \Big)/\alpha$.
\end{proof}
\begin{remark}
The inequality $ x > \sqrt{  4y^2+6 - 2 \sqrt{ 4y^2+9 } } $ is equivalent to 
\begin{equation} 
	\label{ineq:sq-sq}
	\sqrt{ x^2 + 4 } > \sqrt{4y^2 + 9} -1.
\end{equation}
We may check the  condition of Lemma \ref{lemm:An-1} by using this inequality.
\end{remark}
\begin{remark}
Assuming that $x \ge 2y$ implies that \eqref{ineq:sq-sq}.
Thus, in practical way, we take the charge points with $\rho \ge 2R$ using the MFS.
Furthermore, assuming that \eqref{ineq:sq-sq} for $x,y >0$ implies that $x > y$.
The only condition \eqref{ineq:sq-sq}  is sufficient for Lemma \ref{lemm:An-1}.
\end{remark}

Next, we show $R A_n / ( \rho A_{n+1} ) <1$ for $ n \in \{ 0 \} \cup \N$ in the following lemma.
\begin{lemma}\label{lemm:An+1}
For any $x,y \in \R$ with $0< y <  x$ and $n \in  \N $, we have
\begin{align*}
 \frac {y}{x} \frac{ n }{n+1} \frac { I_{n}( y )}{ I_{n+1}( y ) }  \frac{ K_{n}( x )}{ K_{n+1}( x ) } < 1.
\end{align*}
Therefore, replacing $x=\alpha \rho$ and $y=\alpha R$, we have the inequality 
\begin{equation} 
  \frac{ R }{ \rho } \frac {A_n} { A_{n+1}  } < 1, \quad (n \in \{ 0 \} \cup \N).  \label{an/an+1}
\end{equation}
\end{lemma}

\begin{proof}[Proof of Lemma \ref{lemm:An+1}]
We use the following inequalities:
Referring to (\cite[Theorem \ 6]{segura}), we see that 
\begin{equation}\label{segura6}
	\frac{1}{ x } \frac{ K_\nu(x) }{ K_{\nu+1}(x) } \le \frac{1}{ (\nu+ 1 / 2) + \sqrt{ (\nu- 1/2)^2 + x^2 } }, \quad \Big(x>0, \ \nu \ge \frac1 2 \Big),
\end{equation}
and from (\cite[Theorem \ 1.1]{Laforgia}),  we have
\begin{equation}\label{laforgia1.1}
	\frac{ I_{\nu-1}(y) }{I_\nu(y)} < \frac{ \nu + \sqrt{ \nu^2 + y^2 } }{ y }, \quad (y>0, \quad \nu\ge 0).
\end{equation}
By using \eqref{segura6} and \eqref{laforgia1.1}, for any $n \in \N$ we have
\begin{align*}
	& \frac {y}{x} \frac{ n }{n+1} \frac { I_{n}( y )}{ I_{n+1}( y ) }  \frac{ K_{n}( x )}{ K_{n+1}( x ) } \notag\\
	&<  \frac {y}{x} \frac{ n }{n+1} \frac{ n+1 + \sqrt{ (n+1)^2 + y^2 } }{ y } \frac{ x }{  (n+1/2) + \sqrt{ (n-1/2)^2 + x^2} } \notag\\
	&=  \frac{ n }{n+1} \frac{  n+1 + \sqrt{ (n+1)^2 + y^2 }  }{ (n+1/2) + \sqrt{ (n-1/2)^2 + x^2} }\notag\\
	&=:\eta_n (x, y). 
\end{align*}
We denote the above sequence of $n$ by $\{\eta_n (x, y)\}_n $.
We fix $y>0$.
Since $\eta_n (x, y)$ is the monotonically decreasing sequence with respect to $x$ for $x \ge y$, we  show that $\eta_n (y, y) < 1$. 
We have that
\[
\eta_n (y,y) = \frac{ 2 +  p_n }{ 2 + q_n},
\]
where
\begin{align*}
p_n&=\frac1{n+1} \left( \sqrt{(n+1)^2+y^2} - (n+1) \right)\\
&=\frac{ y^2 }{ (n+1 )^2 + (n+1)\sqrt{ (n+1)^2 + y^2} },\\
q_n&= \frac{ 1 }{n} \left( \sqrt{(n-1/2)^2+y^2} -(n-1/2) \right)\\
&=\frac{  y^2 }{ n(n-1/2 ) + n \sqrt{ (n - 1 / 2 )^2 + y^2 } }.
\end{align*}
Comparing the denominators of $p_n$ and $q_n$, we see that $q_n >  p_n$ for any $n \in \N$.
It yields that $\eta_n (x, y) < 1$ for any $x $ and $y$ with $0< y <  x$, and then $RA_n/\rho A_{n+1}  <1$.
\end{proof}

We  show the following lower boundedness.
\begin{lemma}\label{lemma:A_n_tild_A_n}
Assume that $\rho  > \Big( \sqrt{  4\alpha^2 R^2 +6 - 2 \sqrt{ 4\alpha^2 R^2+9 } } \Big)/\alpha$.
Then we have
\begin{align}
	A_{n-1} + A_{n+1} - \frac { 2 R}{\rho} A_n >0, \quad (n \in \N), \label{ineq:A_n}
\end{align}
and 
\begin{align}
		 \tilde{A}_n > 0, \quad (n \in \Z).  \label{ineq:tild_A_n}
\end{align}
\end{lemma}
\begin{proof}[Proof of Lemma \ref{lemma:A_n_tild_A_n}]
Replacing $x = \alpha \rho$ and $y=\alpha R$ and utilizing Lemma \ref{lemm:An-1} and \ref{lemm:An+1},  for $\rho  > \Big( \sqrt{  4\alpha^2 R^2 +6 - 2 \sqrt{ 4\alpha^2 R^2+9 } } \Big)/\alpha$ and $ n \in \N $ we have
\[
A_{n-1}> \frac{ R }{ \rho } A_n, \quad A_{n+1} > \frac{ R }{ \rho } A_n .
\]
Therefore, for any $n \in \N$ we have
\begin{align*}
A_{n-1} + A_{n+1} - \frac{2R}{\rho} A_n 
> 0, 
\end{align*}
which implies \eqref{ineq:A_n}.
Owing to \eqref{ineq:A_n}  we have 
\[
 \tilde{A}_n = \sum_{r=0}^\infty (A_{|n-1|+2r} + A_{|n+1|+2r} -\frac{2R}{\rho} A_{|n|+2r}) > 0 
\]
in the case of $n \neq 0$.
Since 
\begin{align*}
2A_{n+1} - \frac{2R}{\rho} A_n 
> 0, \quad (n \in \{0 \} \cup \N)
\end{align*}
from \eqref{an/an+1}, we see that 
\begin{equation*} 
	 \tilde{A}_0 = \sum_{r=0}^\infty (2A_{1+2r}  -\frac{2R}{\rho} A_{2r}) > 0.
\end{equation*}
Summarizing above, we have \eqref{ineq:tild_A_n}.
\end{proof}

Now we show the proof of Theorem \ref{thm:1}.
\begin{proof}[Proof of Theorem\,\ref{thm:1}]
From \eqref{ineq:tild_A_n}, we have $ \tilde{A}_n >0 $ for any $n \in \Z$.
Thus, from \eqref{eq:f_omega}  we compute that 
\begin{equation*} 
	f(\omega^m) = \frac{N}{R} \sum_{ n \in \Z }  \tilde{A}_{ nN+m } > 0
\end{equation*}
for $m=0,\ldots, N-1$. 
This implies the assertion of Theorem \ref{thm:1}.
\end{proof}

%
%
%

\section{Energy method for the error}\label{sec:NV}
In this section, we provides the algorithm to calculate the error bound for this approximate solution by the energy method.
By using the MFS, we have constructed the approximate solution \eqref{sol:mfs}, which satisfies the boundary condition on the finite $N$ collocation points of $\partial \Omega$.
Set the error function as 
\[
h_N(x)=g(x)-g_N(x).
\]
We note that this term means the error between the solution $g$ and the approximate solution $g_N$.
As both $g$ and $g_N$ satisfy the principal equation of \eqref{eq:gene_prob},
by substituting $h_N$ for the equation and the boundary condition of \eqref{eq:gene_prob}, we have 
\begin{equation}
\label{eq:h}
\left \{ 
\begin{aligned}
&\Delta  h_N -\alpha^2  h_N
=0, \quad x \in \Omega \\
&\dfrac{\partial  h_N}{\partial  n}
= s - \frac{\partial g_N}{\partial n}, \quad  x \in \partial\Omega.
\end{aligned}
\right. 
\end{equation}
To estimate the error bound, we perform a priori estimate.
First we show the following lemmas.
\begin{lemma}
\label{lemm:trace_cir}
Let $\Omega \subset \R^2$ be $B(0,R)$ and $u \in H^1(\Omega)$.
Then, the trace $u|_{\partial \Omega}$ can be interpreted as a function in $L^2(\partial \Omega)$ satisfying
\begin{equation}\label{ineq:trace_cir}
	\left\| u \right\|_{L^2(\partial \Omega)} 
	\le C_{\Omega} \left\| u \right\|_{H^1(\Omega)} ,
\end{equation}
where $C_{\Omega} := (1 + 2/R)$.
\end{lemma}
The proof is put in the Appendix \ref{appen:lemm_trace}.
\begin{lemma}\label{lemm:H^1}
Suppose that $\Omega = B(0,R)$, and $u \in H^2(\Omega)$ satisfies $ \Delta u = \alpha^2 u  $ in $\Omega$.
Then we have 
\begin{equation*} 
	\| u \|_{H^1(\Omega)} \le  \frac{C_2}{ C_{\Omega} } \left\| \frac{\partial u}{ \partial n } \right\|_{L^2(\partial \Omega)},
\end{equation*}
where $C_2 := \min \{ 1, \alpha^2 \}$
\end{lemma}
\begin{proof}[Proof of Lemma \ref{lemm:H^1}]
Multiplying $ \Delta u = \alpha^2 u  $ by $u$ and using the Green formula, we have 
\begin{equation*}
	\alpha^2\| u \|_{L^2(\Omega)}^2 = \int_{\Omega}u \Delta u dx
	 = \int_{\partial \Omega}\frac{\partial u }{\partial n}u dl - \| \nabla u \|^2_{L^2(\Omega)}.
\end{equation*}
Then we obtain that  
\begin{equation*}
	C_2 \| u \|^2_{H^1(\Omega)} 
	\le \| \nabla u \|^2_{L^2(\Omega)}+\alpha^2\| u \|_{L^2(\Omega)}^2
	= \int_{\partial \Omega}\frac{\partial u}{\partial n}u dl,
 \end{equation*}
where $C_{2}=\min \{ 1, \alpha^2 \}$.
Schwarz inequality and the trace operator yields that
\begin{align}
\int_{\partial \Omega}\frac{\partial u}{\partial n}u dl
	&\le \left\|  \frac{\partial u }{\partial n} \right\|_{L^2(\partial \Omega)} \| u \|_{L^2(\partial\Omega)} \notag\\
	&\le C_{\Omega}\left\|  \frac{\partial u }{\partial n} \right\|_{L^2(\partial \Omega)} \| u \|_{H^1(\Omega)} \notag\\
	&\le  \frac{C_{\Omega}^2}{2C_{2} }\left\|  \frac{\partial u }{\partial n} \right\|_{L^2(\partial \Omega)}^2 
	+ \frac{C_{2} }{2} \| u \| ^2_{H^1(\Omega)}, \notag
\end{align}
where $C_{\Omega} = (1 + 2/R)^{1/2}$ is a positive constant specified from the trace operator in $\Omega$ as shown in Lemma \ref{lemm:trace_cir}.
Finally, we see that 
\begin{equation*} 
\| u \|^2_{H^1(\Omega)} 
\le  \frac{ C_{ \Omega } ^2 } { C_{2}^2 } \left\|  \frac{\partial u }{\partial n} \right\|_{L^2(\partial \Omega)}^2.
\end{equation*}

\end{proof}

Next, we show following lemmas.
\begin{lemma}\label{lemm:H^2}
Suppose that $\Omega \subseteq \R^2$ is an arbitrary domain, and $u \in H^2(\Omega)$.
Then we have
\begin{equation*} 
	\| u \|^2_{H^2(\Omega)} 
	\le \| u \|^2_{H^1(\Omega)} + \| u_x \|^2_{H^1(\Omega)} + \| u_y \|^2_{H^1(\Omega)}.
\end{equation*}
\end{lemma}
\begin{proof}[Proof of Lemma \ref{lemm:H^2}]
\begin{align*}
	\text{LHS}
	&=\| u_{xx} \|_{L^2(\Omega)}^2 + \|  u_{yy} \|_{L^2(\Omega)}^2 + \|  u_{xy} \|_{L^2(\Omega)}^2+ \| \nabla u \|_{L^2(\Omega)}^2+ \|  u \|_{L^2(\Omega)}^2,\\
	\text{RHS}
	&=\| u_{xx} \|_{L^2(\Omega)}^2 + \|  u_{yy} \|_{L^2(\Omega)}^2 + 2 \|  u_{xy} \|_{L^2(\Omega)}^2+ 2 \| \nabla u \|_{L^2(\Omega)}^2+ \|  u \|_{L^2(\Omega)}^2.
\end{align*}
\end{proof}
\begin{lemma}\label{lemm:H^2_boun}
Suppose that $\Omega = B(0,R)$, and $u \in H^3(\Omega)$ satisfies $ \Delta u = \alpha^2 u  $ in $\Omega$.
Then we have
\begin{equation*} 
	\| u \|^2_{H^2(\Omega)} 
	\le C_3 \left \{  \left\|  \frac{ \partial u }{\partial n} \right\|_{L^2(\partial \Omega)}^2    
	+ \left\|  \frac{ \partial u_x }{\partial n} \right\|_{L^2(\partial \Omega)}^2 
	+  \left\|  \frac{ \partial u_y }{\partial n} \right\|_{L^2(\partial \Omega)}^2 \right\},
\end{equation*}
where $C_3 : = C_{ \Omega } ^2 / C_{2}^2 $.
\end{lemma}
\begin{proof}[Proof of Lemma \ref{lemm:H^2_boun}]
Applying Lemma \ref{lemm:H^1} to $u, u_x$ and $u_y$, we have 
\begin{align*}
	\| u \|^2_{H^1(\Omega)} 
	&\le C_3 \left\|  \frac{\partial u }{\partial n} \right\|_{L^2(\partial \Omega)}^2,\\
	\| u_x \|^2_{H^1(\Omega)} 
	&\le C_3 \left\|  \frac{ \partial u_x }{\partial n} \right\|_{L^2(\partial \Omega)}^2,\\
	\| u_y \|^2_{H^1(\Omega)} 
	&\le C_3 \left\|  \frac{ \partial u_y }{\partial n} \right\|_{L^2(\partial \Omega)}^2.
\end{align*}
Combining these inequalities and Lemma \ref{lemm:H^2}, we have the assertion of this lemma.
\end{proof}

Applying Lemma \ref{lemm:H^2_boun} to $h_N$ in \eqref{eq:h}, we obtain that
\begin{align}
	\| h_{N} \|^2_{H^2(\Omega)} 
	\le C_3 \left (  \left\|  \frac{ \partial h_{N} }{\partial n} \right\|_{L^2(\partial \Omega)}^2    + \left\|  \frac{ \partial h_{N,x} }{\partial n} \right\|_{L^2(\partial \Omega)}^2 +  \left\|  \frac{ \partial h_{N,y} }{\partial n} \right\|_{L^2(\partial \Omega)}^2 \right ) \label{est:H^2}.
\end{align}
As the function $s$ is given,  the error in $H^2(\Omega)$ is given by these boundary integrations.
Eventually, the Sobolev embedding theorem yields that the error $h_N$ belonging to $C(\Omega)$ can be bounded by \eqref{est:H^2}.

\begin{remark}
If the exact value of $C_\Omega$ is obtained in the general Jordan domain with smooth boundary, the above algorithm for the error to calculate can be extended in the general Jordan domain with smooth boundary.
\end{remark}

\section{Estimates of the Fourier coefficents}\label{sec:proof2}
From the algorithm in previous section the  error bound between the exact solution $g$ and the approximate solution $g_N$ in $H^2(\Omega)$ is given by  the boundary integrations.
In this section 
we will prepare some lemmas before  the  proof of Theorem \ref{thm:2}. 

\subsection{ Upper and lower bounds of the Fourier coefficient  ${\tilde A}_n$}
First, we will estimate the upper and lower bounds of the Fourier coefficient  ${\tilde A}_n$ for $c(\theta)$. 
We use the following inequality \cite{Gaunt}:
\begin{equation*}
	\frac{ 2^{n-1}\Gamma(n)e^{-x} }{ x^n } < K_n(x) < \frac{2^{n-1} \Gamma(n) }{x^n}, \quad (x>0, \quad n\in \N).
\end{equation*}
Also the integration form of the modified Bessel function of the first kind
\[
I_{n}(z) = \frac{ ( z/2)^n  }{\Gamma( n + 1/2)\Gamma( 1/2)} \int_{-1}^1 e^{-zt}(1-t^2 )^{n - \frac12}dt, \quad ( n \in \{0 \} \cup \N )
\]
yields that 
\begin{equation*}
	\frac{ ( y / 2 )^n e^{-y} }{n!} < I_n(y) < \frac{ ( y / 2 )^n e^{y} }{n!}.
\end{equation*}
Combining these inequalities, we have
\begin{equation}\label{An:u-l}
	\frac 12 \Big( \frac{R}{\rho} \Big)^{n+1} e^{-\alpha(\rho + R)} < A_n < \frac 12 \Big( \frac{R}{\rho} \Big)^{n+1} e^{\alpha R}, \quad ( n \in \{0 \} \cup \N).
\end{equation}

We show the following the formulas and  the lower bounds.
\begin{lemma}\label{lemm:A-n_symm}
\begin{align}
		&\tilde{A}_n = \tilde{A}_{-n}, \quad (n \in \Z), \label{AN:symm1}\\
		&\sum_{l\in \Z} {\tilde A}_{lN+n}  = \sum_{l\in \Z} {\tilde A}_{lN-n}, \quad (n \in \Z). \notag
\end{align}
\end{lemma}
\begin{proof}[Proof of Lemma \ref{lemm:A-n_symm}]
We compute that
\begin{align*}
\sum_{l\in \Z} {\tilde A}_{lN+n}
=\sum_{ l\in \Z} {\tilde A}_{ -lN+n}
=\sum_{ l\in \Z} {\tilde A}_{ lN-n}.
\end{align*}
We changed the variable as $l=-l$ in the first equation, and used \eqref{AN:symm1} in the second equation.
\end{proof}
\begin{lemma}\label{lemm:A-n_lower}
Under the assumption of Theorem \ref{thm:1},
there exist a positive constant $C_4 $  independent of $N$ such that 
\begin{equation}\label{A_til_p:l}
	\tilde{A}_n  \ge C_4 \Big(  \frac R \rho  \Big)^{ |n| }, \quad (n \in \Z),
\end{equation}
and
\begin{equation}\label{A_til:l_lN+n}
 \sum_{l\in \Z} {\tilde A}_{lN+n}  > C_4 \Big( \frac R \rho\Big)^{  \min\{ N-n, n \} }, \quad (0 \le n \le N),
\end{equation}
and 
\begin{equation}\label{A_til:l_p}
\sum_{l\in \Z} {\tilde A}_{lN+n} >   C_4  \Big( \frac R \rho\Big)^{ \frac N 2}, \quad (n \in \N).
\end{equation}
\end{lemma}
\begin{proof}[Proof of Lemma \ref{lemm:A-n_lower}]
We compute $\tilde{A}_n $ downward.
Owing to \eqref{AN:symm1}, we firstly estimate $\tilde{A}_n $ for $n \in \N$.
From \eqref{tA:natural}, \eqref{an/an-1}, \eqref{an/an+1} and \eqref{An:u-l} we compute that 
\begin{align*}
	\tilde{A}_n 
	&= \sum_{r=0}^\infty (A_{  n-1 + 2r } + A_{  n+1  +2r} - \frac{2R}{\rho} A_{ n +2r } )\\
	&= \sum_{r=0}^\infty (A_{  n-1 + 2r } - \frac{R}{\rho} A_{ n +2r } )  + \sum_{r=0}^\infty (  A_{  n+1  +2r}  - \frac{R}{\rho} A_{ n +2r }  )\\
	&> A_{  n-1  } - \frac{ R } { \rho } A_{ n  } \\
	& = A_{  n-1  } \Big(  1 - \frac{R  }{  \rho  } \frac{   A_{ n  } }{  A_{  n-1  }  }  \Big) \\
	&\ge C_5 A_{  n-1  }\\
	& >  \frac{ C_5  e^{-\alpha (\rho + R)}  }{ 2 }  \Big(  \frac{R}{\rho}  \Big)^n,
\end{align*}
where 
\begin{align*}
	C_5
	&:= \inf_{ n \in \N } \Big(  1 - \frac{R  }{  \rho  } \frac{   A_{ n  } }{  A_{  n-1  }  }  \Big)\\
	&= 1 - \sup_{ n \in \N }  \frac{R  }{  \rho  } \frac{   A_{ n  } }{  A_{  n-1  }  }.
\end{align*}
From \eqref{an/an-1}  we see that $C_5 > 0$.
Due to $ \tilde A_0 >0$ from \eqref{ineq:tild_A_n}, setting a positive constant as 
\begin{equation*} 
	C_4:= \min \left\{ \frac {C_5  e^{-\alpha (\rho + R)} } 2 , \tilde A_0 \right\},
\end{equation*}
we have \eqref{A_til_p:l}.
Next, we compute $\sum_{l\in \Z} {\tilde A}_{lN+n}$ for $0\le n \le N$ as follows.
Utilizing \eqref{A_til_p:l}, we obtain that 
\begin{align*}
    \sum_{l\in \Z} {\tilde A}_{lN+n} 
    &>\sum_{l = -1}^0 {\tilde A}_{lN+n} \\
&\ge C_4 \Big(  \Big( \frac R \rho\Big)^{N - n} +   \Big( \frac R \rho\Big)^{n} \Big)\\
&> C_4 \Big( \frac R \rho\Big)^{\min\{ N - n, n\} },
\end{align*}
which implies \eqref{A_til:l_lN+n}.

Finally, since the left hand side of \eqref{A_til:l_p} is invariant for $n \equiv m \ (\text{mod} \ N)$ for $ (m \in \Z)$,
we show \eqref{A_til:l_p} in $0\le n < N$.
Due to $2\min\{ N-n, n\} \le N$ in $0\le n \le N$, \eqref{A_til:l_lN+n} yields that \eqref{A_til:l_p}.
\end{proof}

Next, we show the upper bounds.
\begin{lemma} \label{lemm:tA_up}
There exist positive constants $C_6 $ and $C_7 $ independent of $N$ such that
\begin{align}
	&{\tilde A}_n <  C_6 \Big( \frac R \rho \Big)^{|n|}, \quad (n \in  \Z ), \label{A_til:u_n^2}
\end{align}
and 
\begin{align}\label{A_til:ulN+n}
	  \sum_{l \neq 0} {\tilde A}_{lN+n}
	 &< 2 C_7 \Big( \frac R \rho\Big)^{N-n}, \quad (0 \le  n \le N).
\end{align}
\end{lemma}

\begin{proof}[Proof of Lemma \ref{lemm:tA_up}]
From \eqref{AN:symm1}, we estimate $\tilde{A}_n $ for $n \in \N$ using  \eqref{tA:natural}  and \eqref{An:u-l}  as follows: 
\begin{align*}
\tilde{A}_n 
&= \sum_{r=0}^\infty (A_{ n-1 + 2r} + A_{ n+1 +2r} - \frac{2R}{\rho} A_{ n+2r} )\\
&< \sum_{r=0}^\infty (A_{ n-1+2r } + A_{ n+1+ 2r } )\\
& < \frac { e^{\alpha R} }{ 2 } \Big( \Big( \frac R \rho\Big)^{ n-1 } + \Big( \frac R \rho\Big)^{ n+1 } \Big) \sum_{r=0}^\infty \Big( \frac R \rho\Big)^{  1+2r }\\
& = \frac { e^{\alpha R} }{ 2 } \frac{  R / \rho  }{ 1 - (  R / \rho )^2 } \Big( \Big( \frac R \rho\Big)^{ n-1 } + \Big( \frac R \rho\Big)^{ n+1 } \Big)\\
&  = \frac { e^{\alpha R} }{ 2 }  \frac{ \rho R }{ \rho^2 - R^2 }  \Big(  \frac R \rho +  \frac \rho R  \Big)  \Big( \frac R \rho\Big)^{ n } \\
&= C_6  \Big( \frac R \rho\Big)^{ n },
\end{align*}
where $C_6 := ( e^{\alpha R } / 2  ) ( \rho^2 + R^2  ) / ( \rho^2 - R^2 ) = ( e^{\alpha R } / 2  )  ( 1 + R^2 / \rho^2 )  ( 1 - R^2 / \rho^2 )^{-1} $.
Next, we calculate $\tilde{A}_n $ upward.
Applying \eqref{An:u-l}, we compute that 
\begin{align*}
\tilde{A}_0 
&=  \sum_{r=0}^\infty (2 A_{1+2r} - \frac{2R}{\rho} A_{2r} ) \notag\\
&< 2 \sum_{r=0}^\infty  A_{1+2r}  \notag \\
&<  \sum_{r=0}^\infty \Big( \frac  R \rho \Big)^{2+2r} e^{\alpha R} \notag\\
&= e^{ \alpha R} \Big( \frac R \rho \Big)^2 \frac { 1 } { 1 - ( R / \rho )^2 }\notag \\
&= e^{ \alpha R}  \frac{ R^2 }{ \rho^2 - R^2 }\\
&= \frac{ e^{\alpha R } }{ 2  } \frac{ 2 R^2 }{ \rho^2 - R^2 } \\
& < C_6.
\end{align*}
Therefore, 
by using the symmetry \eqref{AN:symm1}, we obtain   \eqref{A_til:u_n^2}.
We compute  $\sum_{l\in \Z} {\tilde A}_{lN+n}$ upward as follows.
Since 
\begin{align*}
    \sum_{l \neq 0} {\tilde A}_{lN+n} = \sum_{l=1}^{\infty} {\tilde A}_{lN+n}  + \sum_{l=1}^{\infty} {\tilde A}_{lN-n} 
\end{align*}
for $n \in \N$, we calculate the two terms, respectively.
For the first term applying \eqref{A_til:u_n^2} for  $n \in \{ 0 \} \cup \N$, we calculate that 
\begin{align}
\sum_{l=1}^{\infty} {\tilde A}_{lN+n} 
&< C_6  \sum_{l=1}^{\infty} \Big( \frac R \rho\Big)^{ lN+n } \notag\\
&= C_6  \Big( \frac R \rho\Big)^n \frac{ (  R / \rho )^N }{ 1 - (  R / \rho )^N } \notag\\
&<C_6  \frac{ \rho }{ \rho -R } \Big( \frac R \rho\Big)^{N+n} \notag\\
&=C_7 \Big( \frac R \rho\Big)^{N+n}, \notag 
\end{align}
where $C_7 :=C_6   \rho / ( \rho -R )$.
Since $lN-n \ge 0$  for $l \in \N$ and  any $0 \le n \le N$, by using \eqref{A_til:u_n^2} we have 
\begin{align*}
\sum_{l=1}^{\infty} {\tilde A}_{lN-n} 
&< C_6  \sum_{l=1}^{\infty} \Big( \frac R \rho\Big)^{ lN-n }\\
&= C_6   \Big( \frac R \rho\Big)^{-n} \frac{  (  R / \rho )^N }{ 1 -  (  R / \rho )^N }\\
&< C_7 \Big( \frac R \rho\Big)^{N-n}.
\end{align*}
Therefore,
\begin{align*}
 \sum_{l\neq0} {\tilde A}_{lN+n} 
 & < C_7 \Big( \frac R \rho\Big)^N \Big( \Big( \frac R \rho\Big)^{n}+\Big( \frac R \rho\Big)^{-n} \Big)\\
 & \le 2 C_7 \Big( \frac R \rho\Big)^{N-n} 
\end{align*}
\end{proof}

\subsection{Upper bound of the Fourier coefficient $a_n$}
Defining the norm as
\[
\left\| g \right\|_{\infty, r_0} := \sup_{ |x|\le r_0} | g (x) |,
\]
and the  sequence as 
\begin{equation}\label{fn:defi}
 \varphi_n :=  \alpha \Big( 1 + \frac{|n|}{\alpha R} \Big) \Big( \frac{ R }{r_0} \Big)^{|n|} , \quad ( n \in \Z ),
\end{equation}
 we have following lemma.
\begin{lemma}\label{lemma:fn}
Under the assumption of Theorem \ref{thm:2},  we have 
\begin{align}\label{fn:upper}
	|a_n| &<  \ph_n  \left\| g \right\|_{\infty, r_0}  = \alpha \Big( 1 + \frac{|n|}{\alpha R} \Big)   \Big( \frac{ R }{r_0} \Big)^{|n|} \left\| g \right\|_{\infty, r_0}, (n \in \Z).
\end{align}
\end{lemma}
\begin{proof}[Proof of Lemma\,\ref{lemma:fn}]
We refer to \cite{Paris} for  
\begin{equation*}
	\frac{ I_\nu(x) }{ I_\nu(y) } < \Big( \frac{x}{y} \Big)^\nu, \quad \Big( 0<x<y, \ \nu > -\frac 1 2\Big),
\end{equation*}
and  to \cite{Laforgia} for 
\begin{equation*}
	t\frac{  I'_{\nu}(t) }{ I_{\nu}(t) } < \sqrt{ t^2 + \nu^2 }, \quad \Big( \nu \ge - \frac{1}{2} \Big).
\end{equation*}
\eqref{eq:exact_g} yields that 
\[
\frac{ I_n( \alpha r ) }{ \alpha I'_n( \alpha R ) } a_n = \frac{1 }{ 2 \pi  } \int_{-\pi}^\pi g(r e ^{i \theta } ) e ^{ - i n \theta } d\theta.
\]
From the assumption of Theorem \ref{thm:2} the function $g$ can be extended to the neighborhood of $\bar \Omega$,  $B(0, r_0)$ with $r_0 > R$.
Thus,
we have 
\begin{align*}
|a_n| 
& = \Big| \frac{ \alpha I'_n(\alpha R)  }{ I_n(\alpha r_0) } \frac{ 1 }{ 2\pi } \int_{-\pi}^\pi g(r_0 e^{ i \theta } ) e^{- in\theta}d\theta \Big|\\
& \le  \frac{ \alpha I'_n(\alpha R)  }{ I_n(\alpha r_0) } \left\| g \right\|_{\infty, r_0} \\
& =   \frac{ \alpha I_n(\alpha R)  }{ I_n(\alpha r_0) }  \frac{  I'_n(\alpha R)  }{ I_n(\alpha R) } \left\| g \right\|_{\infty, r_0}\\
& \le \alpha   \frac{ \sqrt{ \alpha^2 R^2 + n^2 } }{ \alpha R }    \Big( \frac{ R }{r_0} \Big)^n   \left\| g \right\|_{\infty, r_0}\\
& \le \alpha     \frac{ \alpha R + n }{ \alpha R  }    \Big( \frac{ R }{r_0} \Big)^n   \left\| g \right\|_{\infty, r_0}\\
& = \alpha     \Big( 1 + \frac{n}{\alpha R} \Big)   \Big( \frac{ R }{r_0} \Big)^n   \left\| g \right\|_{\infty, r_0}, \quad (n \in \{ 0 \} \cup \N).
\end{align*}
Due to $ |a_{-n}| =  |a_n| $ for $n \in \Z$ 
 we have 
\begin{align*}
	|a_n| 
	&\le \alpha     \Big( 1 + \frac{|n|}{\alpha R} \Big)    \Big( \frac{ R }{r_0} \Big)^{|n|} \left\| g \right\|_{\infty, r_0}, \quad  (n \in \Z).
\end{align*}
\end{proof}

\subsection{Symmetries and upper bounds of $\ph_n$}
First, we show the following lemma.
\begin{lemma}\label{lemm:sym_phy}
\begin{align}
&\ph_n = \ph_{-n}, \quad (n \in \Z),\label{phiN:symm1}\\
&\sum_{l\in\Z} \ph_{Nl+n} = \sum_{l\in\Z} \ph_{Nl-n},  \quad (n \in \Z). \notag
\end{align}
\end{lemma}
\begin{proof}[Proof of Lemma\,\ref{lemm:sym_phy}]
From the same calculation as that of Lemma \ref{lemm:A-n_symm} we have
\begin{align*}
\sum_{l\in \Z} \ph_{lN+n}
=\sum_{l\in \Z} \ph_{ -lN+n}
=\sum_{l\in \Z} \ph_{ lN-n}.
\end{align*}
We changed the variable as $l=-l$ in the first equation, and used \eqref{phiN:symm1} in the second equation.
\end{proof}

Next we give the  estimation for $\ph_n$ in the following lemma.
\begin{lemma}\label{lemm:ph2}
There exist  positive constants  $C_{8 }$ and  $C_ {9 }$ independent of $N$ such that 
\begin{equation}\label{ph:u_l}
	\sum_{l \in \Z} \ph_l  \le \alpha C_{ 8 },
\end{equation}
and
\begin{equation}\label{ph:u_lN+n}
	\sum_{l\neq0}\ph_{lN+n}   < \alpha C_{9 }  \Big(  1 + \frac {2N}{\alpha R} \Big) \Big( \frac{ R }{ r_0 } \Big)^{N-n},  \quad (0 \le n \le N).
\end{equation}
\end{lemma}
\begin{proof}[Proof of Lemma \ref{lemm:ph2}]
\begin{align*}
\sum_{l \in \Z} \ph_l
& = \ph_0 + 2\sum_{l=1}^{\infty}  \ph_l \\
& = \alpha  
+  2 \alpha  \sum_{l=1}^{\infty}    \Big( 1 + \frac{n}{\alpha R} \Big)   \Big( \frac{ R }{r_0} \Big)^{n}  \\
&=\alpha  
+  2 \alpha  \left\{ \frac{   R / r_0  }{ 1 -   R / r_0  } +  \frac{   R / r_0  }{ \alpha R ( 1 -   R / r_0 )^2  }  \right\} \\
&= \alpha   \Big\{  \frac{r_0+R}{r_0-R} + \frac{1}{\alpha R} \frac{2 r_0 R}{(r_0-R)^2 } \Big\}\\
&\le \alpha  \Big\{ \frac{r_0^2}{(r_0-R)^2} + \frac{1}{\alpha R} \frac{2 r_0^2}{(r_0-R)^ 2 } \Big\}\\
&= \alpha    \frac{r_0^2}{(r_0-R)^2} \Big( 1 + \frac{2}{\alpha R} \Big)\\
&=\alpha C_{ 8 }. 
\end{align*}
where $ C_{8} := \alpha  ( r_0 /(  r_0 - R )  )^2 ( 1 + 2/(\alpha R) )  $.

We estimate $\ph_n$ upward as follows.
From $lN -n\ge0$ for $l \in \N$ and any $0 \le n \le N$, we calculate that
\begin{align*}
\sum_{l\neq0}\ph_{lN+n} 
& = \sum_{l=1}^\infty \ph_{lN+n} + \sum_{l=1}^\infty \ph_{lN-n}\\
& =  \alpha \left\{ \sum_{l=1}^\infty    \Big( 1 + \frac{ lN+n } { \alpha R } \Big)   \Big( \frac{ R }{r_0} \Big)^{lN+n}  + \sum_{l=1}^\infty    \Big( 1+ \frac{ lN-n } { \alpha R  } \Big)  \Big( \frac{ R }{r_0} \Big)^{lN-n}  \right\}\\
&<  2 \alpha    \sum_{l=1}^\infty \Big( 1 + \frac{ lN+N }{\alpha R } \Big) \Big( \frac{ R }{r_0} \Big)^{lN-n}  \\
&= 2 \alpha \Big( \frac{ R }{ r_0 } \Big)^{-n} \left\{ \Big( 1 + \frac{ N } { \alpha R  } \Big) \frac{ (  R / r_0  )^N }{ 1 - (  R / r_0  )^N } +  \frac{ N } { \alpha R  } \frac{ (  R / r_0  )^N }{ \Big( 1 - (  R / r_0 )^N \Big)^2}   \right\} \\
& < \frac{ 2 \alpha r_0^2   }{   (r_0 - R)^2 } \Big( 1+ \frac{ 2N } { \alpha R } \Big) \Big( \frac{ R }{ r_0 } \Big)^{N-n} \\
& = \alpha   C_{ 9 }    \Big( 1 + \frac{ 2N } {\alpha R } \Big) \Big( \frac{ R }{ r_0 } \Big)^{N-n}, 
\end{align*}
where $C_{ 9 }:= 2 r_0^2 /   (r_0 - R)^2 = 2 (1 - R/r_0)^{-2}$.
\end{proof}

\subsection{Fourier series expansion of $ \ds \frac{\partial g_N}{\partial n} $}
Next, set  the  discrete Fourier transformation for the sampling of $N$ collocation points as
\begin{align*}
\hat{s}_n:=\frac1N \sum_{l=0}^{N-1} s_l \omega^{-nl},
\end{align*}
where we recall $s_l = s(R e ^{ \theta_l})$.

\begin{lemma}\label{lemma:gN}
The Fourier series expansion of the normal derivative of the approximate solution $g_N$ on $\partial \Omega$, that is $x=Re^{i \theta}$, is expressed by
\begin{equation}
\frac{\partial g_N}{\partial n} (R e ^{ i \theta })
=  \frac{N}{R}  \sum_{n \in \Z} \frac{ {\hat s}_n }{ f(\omega^n) }  {\tilde A}_n e^{i n \theta}. \label{gn:1sigma}
\end{equation}
Moreover, 
\begin{equation}\label{hat_s_n}
	\hat{s}_n = \sum_{ l \in \Z }  a_{n+Nl}, \quad (n \in \Z).
\end{equation}

\end{lemma}

\begin{proof}[Proof of Lemma\,\ref{lemma:gN}]
We perform the Fourier series expansion of normal derivative $g_N$ by regarding it as the function of $\partial \Omega$:
\begin{align*}
\frac{\partial g_N}{\partial n} (R e^{i \theta})
&=\sum_{k=0}^{N-1} Q_k \frac{\partial }{\partial n} K_0( \alpha | Re^{i \theta} - \rho e^{i \theta_k}|)\\
&=\sum_{k=0}^{N-1} Q_k \left( - \alpha K_1(\alpha |R-\rho e^{-i (\theta - \theta_k) }|) \frac{R - \rho \cos(\theta - \theta_k)}{|R-\rho e^{-i(\theta - \theta_k)}|}  \right)\\
&=\sum_{k=0}^{N-1} Q_k c(\theta - \theta_k).
\end{align*}
Here from \eqref{eq:qk} we have
\begin{align*}
Q_k 
&= \sum _{l=1}^N s_l b_{(-k+l)}\\
&= \frac{1}{N} \sum _{l=1}^N s_l \sum _{m=0}^{N-1} \frac{1}{f\left(\omega ^m \right) \omega ^{m (-k+l)}}\\
&=\frac{1}{N} \sum _{m=0}^{N-1} \frac{ \sum _{l=1}^N s_l e^{-i m \theta_l }} { f \left( \omega^m \right) } e^{im\theta_k -im\theta_ N}\\
&=\frac{1}{N} \sum _{m=0}^{N-1} \frac{ \sum _{l=1}^N s_l e^{-i m \theta_l }} { f \left( \omega^m \right) } e^{im\theta_k }\\
&= \sum_{m=0}^{N-1} \frac{\hat{s}_m}{f(\omega^m)} e^{i m \theta_k}.
\end{align*}
On the other hands, using \eqref{eq:c}, we compute that 
\begin{align*}
c(\theta - \theta_k)
&=\frac{1}{R}  \sum_{ n \in \Z} \tilde{A}_n e^{ i n (\theta - \theta_k)}.
\end{align*}
Substituting both, we have the Fourier series expansion of $\partial g_N / \partial n |_{x \in \partial \Omega}$ as   
\begin{align*}
\frac{\partial g_N}{\partial n} (R e^{i \theta})
&=\sum_{k=0}^{N-1} Q_k c(\theta - \theta_k)\\
&=  \frac1R \sum_{k=0}^{N-1} \sum_{m=0}^{N-1} \sum_{ n \in \Z} \frac{\hat{s}_m}{f(\omega^m)} \tilde{A}_n e^{ i n \theta + i ( m - n )\theta_k}\\
& = \frac{N}{R} \sum_{m=0}^{N-1}  \sum_{ n \in \Z, } \frac{ {\hat s}_m }{ f(\omega^m) } {\tilde A}_{ n N+m} e^{ i ( n N + m ) \theta} \\
&= \frac{N}{R} \sum_{m=0}^{N-1}  \sum_{ \substack {  n \in \Z, \\ n \equiv m \ ( \text{mod} \ N)} } \frac{ {\hat s}_m }{ f(\omega^m) } {\tilde A}_n e^{ i n \theta} \\
 &=  \frac{N}{R} \sum_{m=0}^{N-1} \sum_{ \substack {  n \in \Z, \\ n \equiv m \ ( \text{mod} \ N)} } \frac{ {\hat s}_n }{ f(\omega^n) }  {\tilde A}_n e^{i n \theta}\notag\\
 &= \frac{N}{R}  \sum_{n \in \Z} \frac{ {\hat s}_n }{ f(\omega^n) }  {\tilde A}_n e^{i n \theta}.
\end{align*}

Next, we compute that 
\begin{align*}
\hat{s}_n
&= \frac1N \sum_{l=0}^{N-1} s_l \omega^{-nl}\\
&= \frac1{ N} \sum_{l=0}^{N-1} \sum_{m\in \Z } a_m e^{im\theta_l} e^{-in\theta_l} \\
&= \frac1{ N}  \sum_{m\in \Z } a_m \sum_{l=0}^{N-1} e^{i(m-n)\theta_l} \\
&=  \sum_{m\in \Z } a_m \delta_{m-n\in N\Z} \\
&=  \sum_{ l \in \Z } a_{n+Nl}
\end{align*}
for $n \in \Z$.
\end{proof}

\subsection{Fourier series expansion of $\ds \frac{\partial h_N }{\partial n} $, and  explicit forms and upper bounds of $ L^2(\partial \Omega) $ norm}
Next we give the explicit values of the boundary integrations in the following lemma.
\begin{lemma}\label{lemm:boundary_int}
The Fourier series expansion of the normal derivative of the error $h_N$ on $\partial \Omega$, that is $x=Re^{i \theta}$, is expressed by
\begin{equation}\label{hN:series} 
	\begin{split}
	\frac{\partial h_N }{\partial n} (Re^{i \theta})
	&=  \sum_{n\in \Z} \left(  a_n - \frac N R \frac{ {\hat s}_n }{ f(\omega^n) }  {\tilde A}_n \right)e^{i n \theta} \\
	&= \sum_{n\in \Z} \left( a_n  - \frac{  \sum_{l\in \Z} a_{lN+n} }{ \sum_{l\in \Z} {\tilde A}_{lN+n} }  {\tilde A}_n \right)e^{i n \theta}.
	\end{split}
\end{equation}
The value and upper bound of $ \left\| \partial h_N  / \partial n \right\|_{L^2(\partial \Omega)}^2  $ are given by
\begin{equation*} 
	\left\| \frac{\partial h_N }{\partial n} \right\|_{L^2(\partial \Omega)}^2  
	=  2 \pi R   \sum_{n\in \Z} \frac{  \Big( a_n\sum_{l\in \Z, l\neq0} {\tilde A}_{lN+n}  -  {\tilde A}_n \sum_{l\in \Z, l\neq0} a_{lN+n}  \Big)^2 }{  \Big( \sum_{l\in \Z} {\tilde A}_{lN+n} \Big)^2 },
\end{equation*}
and
\begin{equation} \label{ineq:h_N_boun} 
\begin{split}
	 &\frac{ 1 }{ 4 \pi  R }  \left\| \frac{\partial h_N }{\partial n} \right\|_{L^2(\partial \Omega)}^2 \\
	 & \le   \left\| g \right\|_{\infty, r_0}^2  \sum_{n\in \Z} \frac{ \ph_n^2 \Big( \sum_{l\in \Z, l\neq0} {\tilde A}_{lN+n} \Big)^2  +  {\tilde A}_n^2  \Big( \sum_{l\in \Z, l\neq0} \ph_{lN+n}  \Big)^2 }{ \Big( \sum_{l\in \Z} {\tilde A}_{lN+n} \Big)^2 }, 
\end{split}
\end{equation}
respectively.
Moreover,   assuming that  the series \eqref{hN:series}  is differentiable term by term, we have 
\begin{align*}
	 &\left\|  \frac{\partial h_{N,x} }{\partial n} \right\|_{L^2(\partial \Omega)}^2 + \left\| \frac{\partial h_{N,y} }{\partial n} \right\|_{L^2(\partial \Omega)}^2 \notag\\
	&= \frac{ 4 \pi } { R }     \sum_{n\in \Z, n \neq 0} \frac{ n^2 \Big( a_n\sum_{l\in \Z, l\neq0} {\tilde A}_{lN+n}  -  {\tilde A}_n \sum_{l\in \Z, l\neq0} a_{lN+n}  \Big)^2 }{ \Big( \sum_{l\in \Z} {\tilde A}_{lN+n} \Big)^2 }, 
\end{align*}
and
\begin{equation}\label{ineq:H^1_boun} 
\begin{split}
& \frac{ R }{ 8 \pi  } \left\{  \left\| \frac{\partial h_{N,x} }{\partial n} \right\|_{L^2(\partial \Omega)}^2 +  \left\| \frac{\partial h_{N,y} }{\partial n} \right\|_{L^2(\partial \Omega)}^2 \right\} \\
&\le  \left\| g \right\|_{\infty, r_0}^2 \sum_{n\in \Z,n\neq0} \frac{ n^2 \Big\{ \ph_n^2 \Big( \sum_{l\in \Z, l\neq0} {\tilde A}_{lN+n} \Big)^2  +  {\tilde A}_n^2  \Big( \sum_{l\in \Z, l\neq0} \ph_{lN+n}  \Big)^2  \Big\} }{ \Big( \sum_{l\in \Z} {\tilde A}_{lN+n} \Big)^2 }. 
\end{split}
\end{equation}
\end{lemma}
The possibility of differentiation term by term to \eqref{hN:series} will be guaranteed in the proof of Theorem \ref{thm:2} below.
\begin{proof}[Proof of Lemma \ref{lemm:boundary_int}]
Applying \eqref{eq:f_omega}, \eqref{fn:upper}, \eqref{gn:1sigma} and \eqref{hat_s_n}, 
we  compute the boundary integration  as follows:
\begin{equation}\label{L2:hN}
\begin{split}
\left\| \frac{\partial h_N }{\partial n} \right\|_{L^2(\partial \Omega)}^2
&= R \int_{0}^{2\pi} \left|  s(\theta) - \frac{\partial g_N}{\partial n}(R, \theta) \right|^2 d\theta \\
&= R \int_{0}^{2\pi} \left|   \sum_{n\in \Z} a_n e^{in\theta} - \frac{N}{R}  \sum_{n \in \Z} \frac{ {\hat s}_n }{ f(\omega^n) }  {\tilde A}_n e^{i n \theta} \right|^2 d\theta\\
&=  R \int_{0}^{2\pi} \left|  \sum_{n\in \Z}  \Big( a_n  - \frac{N}{R} \frac{ {\hat s}_n }{ f(\omega^n) }  {\tilde A}_n \Big)e^{i n \theta} \right|^2 d\theta \\
&= 2 \pi R \sum_{n\in \Z} \Big( a_n  - \frac{N}{R} \frac{ {\hat s}_n }{ f(\omega^n) }  {\tilde A}_n \Big)^2 \\
&= 2 \pi R \sum_{n\in \Z} \Big( a_n  - \frac{  \sum_{l\in \Z} a_{lN+n} }{ \sum_{l\in \Z} {\tilde A}_{lN+n} }  {\tilde A}_n \Big)^2 \\
&= 2 \pi R    \sum_{n\in \Z} \frac{ \Big( a_n\sum_{l\in \Z, l\neq0} {\tilde A}_{lN+n}  -  {\tilde A}_n \sum_{l\in \Z, l\neq0} a_{lN+n}  \Big)^2 }{ \Big( \sum_{l\in \Z} {\tilde A}_{lN+n} \Big)^2 }  \\
& \le 4 \pi R    \left\| g \right\|_{\infty, r_0}^2  \sum_{n\in \Z} \frac{  \ph_n^2 \Big( \sum_{l\in \Z, l\neq0} {\tilde A}_{lN+n} \Big)^2  +  {\tilde A}_n^2  \Big( \sum_{l\in \Z, l\neq0} \ph_{lN+n}  \Big)^2 }{ \Big( \sum_{l\in \Z} {\tilde A}_{lN+n} \Big)^2 } 
\end{split}
\end{equation}
It implies \eqref{hN:series} and the assertions with respect to the boundary integration of $\partial h_N / \partial n$.

Furthermore, 
changing the variable $x=R\cos\theta$ and $ y=R\sin\theta$, 
we assume that the series of $ \partial h_N / \partial n$ in \eqref{hN:series} is differentiable term by term.
Then we compute that
\begin{align*} 
	 \frac{ \partial h_{N,x} }{\partial n} 
	 & =   \Big(     \frac{  \partial s }{\partial \theta }   -    \frac{\partial}{\partial \theta } \frac{  \partial g_N }{\partial n}  \Big) \frac{ \partial \theta }{ \partial x }  \\
	 & =  \frac{ -\sin \theta }{R} \Big( i \sum_{n\in \Z} n a_n e^{in\theta} - \frac{iN}{R}  \sum_{n \in \Z} \frac{ n {\hat s}_n }{ f(\omega^n) }  {\tilde A}_n e^{i n \theta}   \Big),\\
	 &= \frac{ - i \sin \theta }{R} \sum_{n\in \Z} n\Big( a_n -   \frac{  N }{R} \frac{  {\hat s}_n }{ f(\omega^n) }  {\tilde A}_n \Big) e^{in\theta} \\
	 &= -\frac{  1 }{2R} \sum_{n\in \Z} n\Big( a_n -   \frac{  N }{R} \frac{  {\hat s}_n }{ f(\omega^n) }  {\tilde A}_n \Big) ( e^{i (n+1) \theta} - e^{i ( n-1) \theta}),
\end{align*}
and that
	\begin{align*}	
	 \frac{ \partial h_{N,y} }{\partial n} 
	 & =   \Big(     \frac{  \partial s }{\partial \theta }   -    \frac{\partial}{\partial \theta } \frac{  \partial g_N }{\partial n}  \Big) \frac{ \partial \theta }{ \partial y }  \\
	 & =  \frac{ \cos \theta }{R} \Big( i \sum_{n\in \Z} n a_n e^{in\theta} - \frac{iN}{R}  \sum_{n \in \Z} \frac{ n {\hat s}_n }{ f(\omega^n) }  {\tilde A}_n e^{i n \theta}   \Big),\\
	 &= \frac{ i \cos \theta }{R} \sum_{n\in \Z} n\Big( a_n -   \frac{  N }{R} \frac{  {\hat s}_n }{ f(\omega^n) }  {\tilde A}_n \Big) e^{in\theta}\\
	 &= \frac{ i  }{2R} \sum_{n\in \Z} n\Big( a_n -   \frac{  N }{R} \frac{  {\hat s}_n }{ f(\omega^n) }  {\tilde A}_n \Big)  ( e^{i (n+1) \theta} + e^{i ( n-1) \theta}).
\end{align*}
Using these calculations and computing boundary integration similarly to \eqref{L2:hN} , we obtain that 
\begin{align*}
	& \left\|  \frac{\partial h_{N,x} }{\partial n} \right\|_{L^2(\partial \Omega)}^2 + \left\| \frac{\partial h_{N,y} }{\partial n} \right\|_{L^2(\partial \Omega)}^2 \\
	 &= \frac{ 4 \pi } { R }    \sum_{n\in \Z}n^2\Big( a_n -   \frac{  N }{R} \frac{  {\hat s}_n }{ f(\omega^n) }  {\tilde A}_n \Big)^2 \\
	 &= \frac{ 4 \pi } { R }    \sum_{n\in \Z, n \neq 0} \frac{ n^2 }{ \Big( \sum_{l\in \Z} {\tilde A}_{lN+n} \Big)^2 } \Big( a_n\sum_{l\in \Z, l\neq0} {\tilde A}_{lN+n}  -  {\tilde A}_n \sum_{l\in \Z, l\neq0} a_{lN+n}  \Big)^2.
\end{align*}
Therefore, we see that
\begin{align*}
& \frac{ R }{ 8 \pi  } \left\{  \left\| \frac{\partial h_{N,x} }{\partial n} \right\|_{L^2(\partial \Omega)}^2 +  \left\| \frac{\partial h_{N,y} }{\partial n} \right\|_{L^2(\partial \Omega)}^2 \right\}\\
&\le  \left\| g \right\|_{\infty, r_0}^2 \sum_{n\in \Z,n\neq0} \frac{ n^2 }{ \Big( \sum_{l\in \Z} {\tilde A}_{lN+n} \Big)^2 } \Big\{ \ph_n^2 \Big( \sum_{l\in \Z, l\neq0} {\tilde A}_{lN+n} \Big)^2  +  {\tilde A}_n^2  \Big( \sum_{l\in \Z, l\neq0} \ph_{lN+n}  \Big)^2  \Big\}.
\end{align*}
\end{proof}

\section{Convergence}\label{sec:conv}
We will explain the proof of Theorem \ref{thm:2} after showing the following Lemma.
\begin{lemma}\label{lemm:convergence}
There exist positive constants $C_{11}$ and $C_{12}$ independent of $N$ and $g$ such that
\begin{equation}\label{conv:1}
	  \left\| \frac{\partial h_N }{\partial n} \right\|_{L^2(\partial \Omega)}^2  
	  \le C_{11}  \tau_N \left\| g \right\|_{\infty, r_0}^2, 
\end{equation}
and 
\begin{align}\label{conv:2}
	  \left\| \frac{\partial h_{N,x} }{\partial n} \right\|_{L^2(\partial \Omega)}^2 +  \left\| \frac{\partial h_{N,y} }{\partial n} \right\|_{L^2(\partial \Omega)}^2
	  \le C_{12}  N^2 \tau_N \left\| g \right\|_{\infty, r_0}^2,
\end{align}
where
\begin{equation*} 
	\tau_N :=  \max\left\{ \Big(  \frac{ R}{\rho}  \Big)^N, N^2 \Big(  \frac{ R }{r_0}  \Big)^N  \right\}.
\end{equation*}
\end{lemma}

\begin{proof}[Proof of Lemma \ref{lemm:convergence}]
Since the error $h_N$ in $H^2(\Omega)$ can be bounded by the boundary integrations in \eqref{est:H^2}, we estimate the values of the boundary integrations.
From  \eqref{ineq:h_N_boun} 
we put the  bound of the boundary integration as 
\begin{align*}
\text{  Right hand side of \eqref{ineq:h_N_boun}}
&=:  \left\| g \right\|_{\infty, r_0}^2 ( E_1 +E_2 ) . \notag
\end{align*}
Furthermore, we divide $E_1$ and $E_2$  into 3 or 4 parts with respect to $n$, respectively. 
Here, we introduce an integer $p$ as the integer part of $N/2$, that is,  $ p = [  N / 2  ]$, where $[\cdot]$ is the Gauss's symbol. 
Then, from  Lemma \ref{lemm:A-n_symm}, and  Lemma \ref{lemm:sym_phy}, we write $E_1$ and $E_2$ as
\begin{align*}
E_1 
&=  \frac{ \Big( \sum_{ l\neq0} {\tilde A}_{lN} \Big)^2   \ph_0^2  }{ \Big( \sum_{ l\in \Z} {\tilde A}_{lN} \Big)^2 }    
 +2 \sum_{n =1}^p \frac{ \Big( \sum_{l\neq0} {\tilde A}_{lN+n} \Big)^2  \ph_n^2 }{ \Big( \sum_{l\in \Z} {\tilde A}_{lN+n} \Big)^2 }      
 + 2 \sum_{n = p+1}^\infty \frac{  \Big( \sum_{l\neq0} {\tilde A}_{lN+n} \Big)^2   \ph_n^2  }{ \Big( \sum_{l\in \Z} {\tilde A}_{lN+n} \Big)^2 }  \\
 &=:  E_{1,1} + E_{1,2} + E_{1,3},
\end{align*}
and
\begin{align*}
E_2 
&= \frac{  {\tilde A}_0^2  \Big( \sum_{l\neq0} \ph_{lN}  \Big)^2 }{ \Big( \sum_{l\in \Z} {\tilde A}_{lN} \Big)^2 } 
+ 2  \sum_{n =1}^p \frac{   {\tilde A}_n^2   \Big( \sum_{l\neq0} \ph_{lN+n}  \Big)^2   }{ \Big( \sum_{l\in \Z} {\tilde A}_{lN+n} \Big)^2 }\\
& \qquad + 2 \sum_{n = p+1}^N \frac{  {\tilde A}_n^2    \Big( \sum_{l\neq0} \ph_{lN+n}  \Big)^2  }{ \Big( \sum_{l\in \Z} {\tilde A}_{lN+n} \Big)^2 } 
+ 2 \sum_{n = N+1}^\infty \frac{  {\tilde A}_n^2   \Big( \sum_{l\neq0} \ph_{lN+n}  \Big)^2  }{ \Big( \sum_{l\in \Z} {\tilde A}_{lN+n} \Big)^2 }  \\
&=: E_{2,1}+  E_{2,2} + E_{2,3} + E_{2,4}.
\end{align*}

First, we estimate $E_{1,1} + E_{1,2}/2$.
For the integer $p=[N/2]$, we see that $ (N-1)/2\le p\le N/2 $ regardless of even and odd of $N$.
Due to $n \le N-n$ for $1 \le n \le p$,
applying \eqref{A_til:l_lN+n} and \eqref{A_til:ulN+n}, we compute that 
\begin{align*}
E_{1,1} + \frac{E_{1,2}}{2} 
&= \sum_{ n = 0 }^p \frac{  \Big( \sum_{l\neq0} {\tilde A}_{lN+n} \Big)^2  \ph_n^2 }{ \Big( \sum_{l\in \Z} {\tilde A}_{lN+n} \Big)^2 } \\
&<  \sum_{n = 0 }^p \frac{1}{ C_4^2 } \Big( \frac{ R }{ \rho }\Big)^{-2n}       \times 4 C_7 ^2 \Big( \frac{ R }{ \rho }\Big)^{2N-2n} \times   \alpha^2 \Big( 1+ \frac{ n }{\alpha R} \Big)^2  \Big( \frac{ R }{ r_0 }\Big)^{2n}   \\
&<  \frac {4 \alpha^2 C_7^2 }{ C_4^2 }  \Big( 1+ \frac{ p }{\alpha R} \Big)^2   \sum_{ n  = 0 }^p    \Big( \frac{ R }{ \rho }\Big)^{2N - 4n} \Big( \frac{ R  }{ r_0 }\Big)^{2n}    \\
&\le
\left\{
\begin{aligned}
&   \frac {4 \alpha^2 C_7^2 }{ C_4^2 } \frac{ Rr_0 }{  Rr_0 - \rho^2 }   \Big( 1+ \frac{ N } { 2\alpha R} \Big)^2 \Big( \frac{ R }{ \rho }\Big)^{2N}, \quad     \Big(     \frac{ R^2 }{\rho^2} >   \frac {R} {  r_0  }   \Big),   \\[2mm]
&    \frac {4 \alpha^2 C_7^2 }{ C_4^2 }   \frac{N}{2} \Big( 1+ \frac{ N }{ 2 \alpha R} \Big)^2 \Big( \frac{ R }{ \rho }\Big)^{2N} , \quad    \Big(     \frac{ R^2 }{\rho^2}  =  \frac {R} {  r_0  }   \Big),\\[2mm]
&  \frac {4 \alpha^2 C_7^2 }{ C_4^2 }  \frac{ \rho^2 }{ \rho^2 - Rr_0 }   \Big( 1+ \frac{ N } { 2 \alpha R} \Big)^2 \Big( \frac{ R }{ r_0 }\Big)^{N}, \quad     \Big(     \frac{ R^2 }{\rho^2} <   \frac {R} {  r_0  }   \Big).\\
\end{aligned}
\right.
\end{align*}
Setting 
\[
C_{10} : = \sup_{N \in \N} N^3 \Big(   \frac{R}{\rho}   \Big)^N \le \Big(   \frac{ 3 }{ \log \rho - \log R }  \Big)^3 e^{-3},
\]
we obtain that  
\[
N^3 \Big(  \frac{R}{\rho}  \Big)^{2N} \le C_{10} \Big(   \frac{R}{\rho}   \Big)^N, \quad (N\in \N).
\]
Thus, we see that $E_{1,1} + E_{1,2}/2$ converges to $0$ in the order of $\max \{  (R/\rho)^{N},  N^2(R/r_0)^N\}$.

For $E_{2,1}+E_{2,2}/2$, by using \eqref{ph:u_lN+n} we compute that 
\begin{align*}
E_{2,1} + \frac{E_{2,2}} 2
& = \sum_{n =0}^p \frac{  {\tilde A}_n^2 }{ \Big( \sum_{l\in \Z} {\tilde A}_{lN+n} \Big)^2 }    \Big( \sum_{l\neq0} \ph_{lN+n}  \Big)^2 \\
&<     \sum_{n =0}^p   \Big( \sum_{l\neq0} \ph_{lN+n}  \Big)^2\\
&<  \alpha^2  C_{ 9 }^2   \Big( 1 + \frac{ 2N } { \alpha R } \Big)^2 \Big( \frac{ R }{ r_0 } \Big)^{2N}  \sum_{n =0}^p  \Big( \frac{ R }{ r_0 } \Big)^{ -2n }\\
&=  \alpha^2     C_{ 9 }^2    \Big( 1 + \frac{ 2N } { \alpha R } \Big)^2 \Big( \frac{ R }{ r_0 } \Big)^{2N}  \frac{    1- (  r_0 / R  )^{ 2p+2 }     }{ 1 - (  r_0 / R  )^2 } \\
& < \alpha^2    C_{ 9 }^2   \frac {  1  } {  (  r_0 / R  )^2  - 1  }  \Big(1 + \frac{ 2N } { \alpha R }  \Big)^2 \Big( \frac{ R }{ r_0 } \Big)^{2N} \Big(  \Big( \frac{ r_0 }{ R } \Big)^{ N + 2 } -1 \Big)\\
&<     \alpha^2      C_{ 9 }^2      \Big(  1 - \frac R r_0  \Big)^{-1}     \Big( 1 + \frac{ 2N } { \alpha R }  \Big)^2 \Big( \frac{ R }{ r_0 } \Big)^{N}
\end{align*}
Thus,  $E_{2,1}+E_{2,2}/2$ converges to $0$ in $N^2(R/r_0)^N$ order.

Next, we estimate $E_{1,3}$.
Since ${\tilde A}_n>0$ for any $ n \in \Z$, and $ (N-1)/2\le p \le N/2 $,   we compute that 
\begin{equation}\label{est:T_1}
\begin{split}
\frac{E_{1,3}}{2}
& =\sum_{n = p+1}^\infty \frac{ \Big( \sum_{l\neq0} {\tilde A}_{lN+n} \Big)^2  \ph_n^2  }{ \Big( \sum_{l\in \Z} {\tilde A}_{lN+n} \Big)^2 }   
<  \sum_{n = p+1}^\infty  \ph_n^2 
< \Big( \sum_{n = p+1}^\infty  \ph_n \Big)^2\\
& = \alpha^2  \Big( \sum_{n = 1}^\infty  \Big( 1 + \frac{ n+p }{ \alpha R } \Big) \Big( \frac{ R }{ r_0 } \Big)^{n+p} \Big)^2\\
&= \alpha^2  \left[ \Big( \frac{ R }{ r_0 } \Big)^p \left\{  \Big( 1+ \frac{ p }{ \alpha R } \Big)\frac{   R / r_0   }{ 1 -   R / r_0   } + \frac{1 }{ \alpha R} \frac{   R / r_0   }{ ( 1 -   R / r_0  )^2 }  \right\}  \right]^2  \\
& < \alpha^2  \Big(  \frac{ R }{ r_0 } \Big)^{N-1} \Big\{ \Big( 1+\frac N {2 \alpha R}  \Big) \frac{  R/  r_0 }{ ( 1 -   R / r_0 )^2} + \frac{1}{\alpha R} \frac{   R / r_0 }{ ( 1 -   R / r_0 ) ^2}  \Big\}^2\\
& =    \frac{    \alpha^2    C_9^2    }{4}   \frac{R}{r_0}   \Big( 1+ \frac{N+2}{2\alpha R} \Big)^2   \Big( \frac{ R }{ r_0 }  \Big)^{N}
\end{split}
\end{equation}
from \eqref{fn:defi}.
Therefore，$E_{1,3}$ converges to $0$ in  $N^2( R/r_0 )^{N}$ order.

Next we compute $E_{2,3}$.
From $n= p+1, \ldots, N$ we note that $n > N-n$, and thus, $(R/\rho)^n < (R/\rho)^{N-n}$.
Utilizing \eqref{A_til:l_lN+n}, \eqref{A_til:u_n^2} and \eqref{ph:u_lN+n}, we calculate that
\begin{align*}
\frac{E_{2,3}}{2} 
&= \sum_{n = p+1}^N \frac{ {\tilde A}_n^2 }{ \Big( \sum_{l\in \Z} {\tilde A}_{lN+n} \Big)^2 } \Big( \sum_{l\neq0} \ph_{lN+n}  \Big)^2 \\
 & <  \sum_{n = p+1}^N \frac{ 1  }{ C_4^2 } \Big(  \frac{ R }{ \rho } \Big)^{ 2n-2N } \times C_6^2 \Big(  \frac{ R }{ \rho }  \Big)^{2n} \times \alpha^2 C_{9 }^2  \Big( 1 + \frac{ 2N } { \alpha R } \Big)^2 \Big( \frac{ R }{ r_0 } \Big)^{2N-2n}\\
 & =   \frac{ \alpha^2   C_6^2 C_{ 9 }^2  }{ C_4^2 }    \Big( 1 + \frac{ 2N } { \alpha R } \Big)^2    \sum_{n = 0 }^ { N - p -1   }    \Big(  \frac{ R  }{ \rho } \Big)^{2N - 4 n }   \Big( \frac{ R }{ r_0 } \Big)^{2n}  \\
 &  \le   \frac{ \alpha^2   C_6^2 C_{ 9 }^2  }{ C_4^2 }    \Big( 1 + \frac{ 2N } { \alpha R } \Big)^2     \sum_{ n = 0 }^p    \Big(  \frac{ R  }{ \rho } \Big)^{2N - 4 n }   \Big( \frac{ R }{ r_0 } \Big)^{2n}     \\
&\le
\left\{
\begin{aligned}
&    \frac{ \alpha^2  C_6^2 C_{ 9 }^2  }{ C_4^2 }    \frac{  r_0 R  }{ r_0 R - \rho^2  }  \Big( 1 + \frac{ 2N } { \alpha R } \Big)^2 \Big(  \frac{ R }{ \rho } \Big)^{2N}  , \quad     \Big(     \frac{ R^2 }{\rho^2} >   \frac {R} {  r_0  }   \Big).\\[2mm]
&      \frac{ \alpha^2  C_6^2 C_{ 9 }^2   }{ C_4^2 }      \frac{N}{2}  \Big( 1 + \frac{ 2N } { \alpha R } \Big)^2  \Big(  \frac{ R }{ \rho } \Big)^{2N} ,  \quad    \Big(     \frac{ R^2 }{\rho^2} =   \frac {R} {  r_0  }   \Big),\\
&    \frac{ \alpha^2 C_6^2 C_{ 9 }^2  }{ C_4^2 }      \frac{ \rho^2  }{ \rho^2 - r_0 R }  \Big( 1 + \frac{ 2N } { \alpha R } \Big)^2 \Big(  \frac{ R }{ r_0 } \Big)^{N}  , \quad     \Big(     \frac{ R^2 }{\rho^2}  <   \frac {R} {  r_0  }   \Big),   \\[2mm]
\end{aligned}
\right.
\end{align*}
Using $C_{10}$ similarly to the estimation of $E_{1.1}+E_{1,2}/2$,
we show that $E_{2,3}$ converges to $0$ in the order of $\max\{ (R/\rho)^{N} ,   N^2(R/r_0)^{N} \}$.

Finally, applying \eqref{A_til:l_p}, \eqref{A_til:u_n^2} and \eqref{ph:u_l}, we have 
\begin{equation}\label{est:Y_2}
\begin{split}
\frac{E_{2,4}}{2} 
&=  \sum_{n = N+1}^\infty \frac{ {\tilde A}_n^2 }{ \Big( \sum_{l\in \Z} {\tilde A}_{lN+n} \Big)^2 } \Big( \sum_{l\neq0} \ph_{lN+n}  \Big)^2 \\
&<  \sum_{n = N+1}^\infty \frac{ {\tilde A}_n^2 }{ \Big( \sum_{l\in \Z} {\tilde A}_{lN+n} \Big)^2 } \Big( \sum_{l\in \Z} \ph_{l}  \Big)^2 \\
&< \sum_{n = N+1}^\infty  \frac{ 1 }{ C_4^2 } \Big( \frac R \rho\Big)^{ -N} \times C_6^2 \Big(  \frac{ R }{ \rho }  \Big)^{2n} \times  \alpha^2 C_{8}^2   \\
& = \frac{  \alpha^2   C_6^2 C_{8}^2 }{ C_4^2 }     \Big( \frac R \rho\Big)^{ -N} \sum_{n = N+1}^\infty  \Big(  \frac{ R }{ \rho }  \Big)^{2n}\\
& = \frac{  \alpha^2   C_6^2 C_{8}^2 }{ C_4^2 }     \Big( \frac R \rho\Big)^{ -N} \frac{ (  R / \rho )^{2(N+1)} }{ 1 - (  R / \rho )^2 } \\
& = \frac{  \alpha^2   C_6^2 C_{8}^2 }{ C_4^2 }   \Big(  1  -   \frac {R^2} {\rho^2}  \Big)^{-1}    \Big( \frac R \rho\Big)^{ N +2  }.  
\end{split}
\end{equation}
We see that $E_{2,4}$ converges to $0$ in  $(R/\rho)^{N}$ order.
Summarizing above all, we obtain \eqref{conv:1}.

Next 
we will estimate $  \left\| \partial h_{N,x} / \partial n  \right\|_{  L^2( \partial \Omega)}^2 +  \left\| \partial h_{N,y} / \partial n  \right\|_{  L^2( \partial \Omega)}^2 $ in \eqref{est:H^2}.
Recalling \eqref{ineq:H^1_boun}, we put the bound as
\begin{align*}
\text{Right hand side of \eqref{ineq:H^1_boun}}
&=: \left\| g \right\|_{\infty, r_0}^2 (E'_1 +E'_2). \notag
\end{align*}
We write $E'_1$ and $E'_2$ as
\begin{align*}
E'_1 
&=  
 2 \sum_{n =1}^p \frac{ n^2 \Big( \sum_{l\neq0} {\tilde A}_{lN+n} \Big)^2  \ph_n^2 }{ \Big( \sum_{l\in \Z} {\tilde A}_{lN+n} \Big)^2 }      
 + 2 \sum_{n = p+1}^\infty \frac{ n^2 \Big( \sum_{l\neq0} {\tilde A}_{lN+n} \Big)^2   \ph_n^2  }{ \Big( \sum_{l\in \Z} {\tilde A}_{lN+n} \Big)^2 }  \\
 &=:    E_{1,2}' +E_{1,3}',
\end{align*}
and 
\begin{align*}
E'_2
&= 2  \sum_{n =1}^p \frac{  n^2 {\tilde A}_n^2   \Big( \sum_{l\neq0} \ph_{lN+n}  \Big)^2   }{ \Big( \sum_{l\in \Z} {\tilde A}_{lN+n} \Big)^2 }\\
 &\qquad + 2 \sum_{n = p+1}^N \frac{ n^2 {\tilde A}_n^2    \Big( \sum_{l\neq0} \ph_{lN+n}  \Big)^2  }{ \Big( \sum_{l\in \Z} {\tilde A}_{lN+n} \Big)^2 } 
+ 2 \sum_{n = N+1}^\infty \frac{ n^2 {\tilde A}_n^2   \Big( \sum_{l\neq0} \ph_{lN+n}  \Big)^2  }{ \Big( \sum_{l\in \Z} {\tilde A}_{lN+n} \Big)^2 }  \\
&=:   E_{2,2}' + E_{2,3}' + E_{2,4}',
\end{align*}
respectively. 
As we see that $E_{1,2}' \le (N/2)^2 E_{1,2} $, $E_{2,2}' \le (N/2)^2 E_{2,2} $, and $E_{2,3}' \le N^2 E_{2,3}$, we estimate $E_{1,3}'$ and $E_{2,4}'$.
We use the following formula 
\[
\sum _{n=1}^{\infty } n^2 r^n  =  \frac{r (r+1)}{(1-r)^3} \le \frac{ 2 r  }{(1-r)^3}, \quad (0<r<1).
\]
Similarly to \eqref{est:T_1}, we have
\begin{equation*}
\begin{split}
\frac{E_{1,3}'}{2}
& =\sum_{n = p+1}^\infty \frac{ n^2\Big( \sum_{l\neq0} {\tilde A}_{lN+n} \Big)^2  \ph_n^2  }{ \Big( \sum_{l\in \Z} {\tilde A}_{lN+n} \Big)^2 }  
<  \sum_{n = p+1}^\infty  n^2 \ph_n^2 
< \Big( \sum_{n = p+1}^\infty n \ph_n \Big)^2\\
& = \alpha^2  \Big( \sum_{n = 1}^\infty  (n+p) \Big( 1 + \frac{ n+p }{ \alpha R } \Big) \Big( \frac{ R }{ r_0 } \Big)^{n+p} \Big)^2\\
&= \alpha^2  \left[ \Big( \frac{ R }{ r_0 } \Big)^p \left\{  \Big( p+ \frac{ p^2 }{ \alpha R } \Big) \frac{   R / r_0   }{ 1 -   R / r_0   } \right. \right. \\
&\qquad \qquad \qquad \left. \left. +  \Big( 1+ \frac{ 2 p }{ \alpha R } \Big)   \frac{   R / r_0   }{ ( 1 -   R / r_0  )^2 } 
+ \frac{1 }{ \alpha R }   \frac{   R / r_0   ( 1 +   R / r_0  ) }{ ( 1 -   R / r_0  )^3 }   \right\}  \right]^2  \\
& <  \frac{ \alpha^2  C_9^3 }{  8  }   \frac{ R }{r_0}      \Big( 1+ \frac N 2 + \frac{N^2 + 4N + 8}{4\alpha R} \Big)^2   \Big( \frac{ R }{ r_0 }  \Big)^{N}.
\end{split}
\end{equation*}
Therefore, $E_{1,3}'$ converges to $0$ in  $N^4( R/r_0 )^{N}$ order.
Further, similarly to \eqref{est:Y_2}, we compute that 
\begin{align*}
\frac{E_{2,4}'}{2} 
&=  \sum_{n = N+1}^\infty \frac{ n^2 {\tilde A}_n^2 }{ \Big( \sum_{l\in \Z} {\tilde A}_{lN+n} \Big)^2 } \Big( \sum_{l\neq0} \ph_{lN+n}  \Big)^2 \\
&<  \sum_{n = N+1}^\infty \frac{  n^2 {\tilde A}_n^2 }{ \Big( \sum_{l\in \Z} {\tilde A}_{lN+n} \Big)^2 } \Big( \sum_{l\in \Z} \ph_{l}  \Big)^2 \\
&< \sum_{n = N+1}^\infty  \frac{ n^2 }{ C_4^2 } \Big( \frac R \rho\Big)^{ -N} \times C_6^2 \Big(  \frac{ R }{ \rho }  \Big)^{2n} \times  \alpha^2   C_{ 8 }^2   \\
& = \frac{  \alpha^2   C_6^2 C_{8}^2 }{ C_4^2 }     \Big( \frac R \rho\Big)^{ -N} \sum_{n = N+1}^\infty  n^2 \Big(  \frac{ R }{ \rho }  \Big)^{2n}\\
& = \frac{  \alpha^2   C_6^2 C_{8}^2 }{ C_4^2 }     \Big( \frac R \rho\Big)^{ N}   \left\{     \frac{ (  R / \rho )^ 2 ( 1 + (  R / \rho )^ 2 )  }{ ( 1 - (  R / \rho )^2 )^3 } + 2N   \frac{ (  R / \rho )^ 2  }{ ( 1 - (  R / \rho )^2 )^2 } + N^2  \frac{ (  R / \rho )^ 2 }{  1 - (  R / \rho )^2  } \right\} \\
& < \frac{  \alpha^2   C_6^2 C_{8}^2 }{ C_4^2 }   \Big(  1  -   \frac {R^2} {\rho^2}  \Big)^{-3}       (N^2 + 2N + 2)  \Big( \frac R \rho\Big)^{ N + 2 }.  
\end{align*}
We see that $E_{2,4}'$ converges to $0$ in  $N^2(R/\rho)^{N}$ order.
Consequently, the series of $ \partial h_N / \partial n$ is differentiable term by term.
Summarizing above all, we obtain \eqref{conv:2}.
\end{proof}
\begin{proof}[Proof of Theorem\,\ref{thm:2}]
Finally,  setting 
\[
a: = \frac{ R }{ \min\{ \rho, r_0 \} },
\]
and using \eqref{est:H^2}, Lemma \ref{lemm:convergence}, and  the Sobolev embedding theorem, we have the assertion of Theorem \ref{thm:2}.
\end{proof}

\section{Numerical simulation}\label{sec:simu}
We numerically investigate the error in \eqref{est:H^2} against  the number of the collocation points $N$.
We denote the error in \eqref{est:H^2} as 
\begin{align*}
	 F(N)
	 &:=C_3 \left (  \left\|  \frac{ \partial h_{N} }{\partial n} \right\|_{L^2(\partial \Omega)}^2   
	  + \left\|  \frac{ \partial h_{N,x} }{\partial n} \right\|_{L^2(\partial \Omega)}^2 
	  +  \left\|  \frac{ \partial h_{N,y} }{\partial n} \right\|_{L^2(\partial \Omega)}^2 \right) 
\end{align*}
for the simple description in this section.
From $\partial h_N / \partial n (R e^{ i \theta} )=S (\theta) - \partial g_N / \partial n (R e^{ i \theta} )$, we can compute the above error function $F(N)$.
\begin{figure*}[bt]
\begin{center}
\begin{tabular}{ccc}
\includegraphics[width=3.9cm, bb=0 0 180 126]{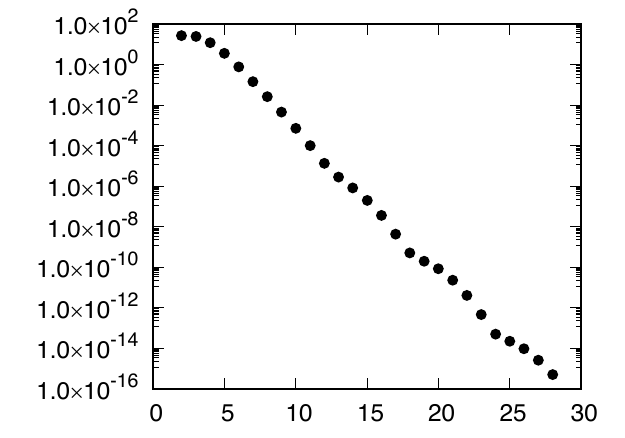} &
\includegraphics[width=3.9cm, bb=0 0 180 126]{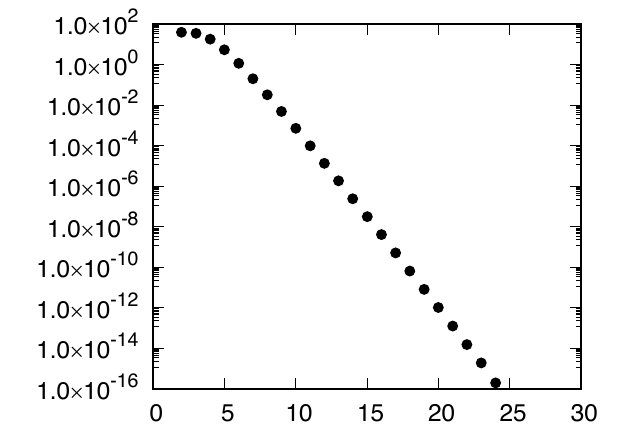} &
\includegraphics[width=3.9cm, bb=0 0 180 126]{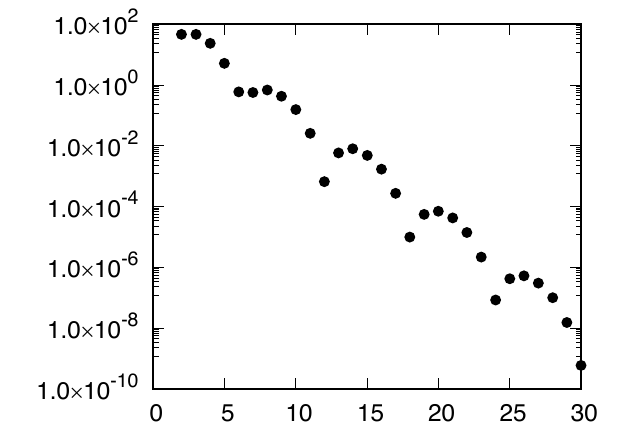} \\
(a) & (b) & (c)
\end{tabular}
\end{center}
\caption{
The numerical results of the error $F(N)=C_3  ( \| \partial h_N / \partial n  \|_{L^2(\partial \Omega)}^2 +  \| \partial h_{N,x} / \partial n  \|_{L^2(\partial \Omega)}^2 +  \| \partial h_{N,y} / \partial n  \|_{L^2(\partial \Omega)}^2 ) $ against $N$ with the parameters same as  Fig.\ \ref{fig:MFS} (b).
The vertical and horizontal axes correspond to  $ F(N)  $ and $N$, respectively.
The vertical axis is plotted on the logarithmic scale. 
(a) $\Phi(r)=e^{-\alpha r}/\sqrt{r} $ and  $P=0.2 e^{i \pi / 3}$,  (b) $\Phi(r)=e^{-r^2}$  and  $P=0.2 e^{i \pi / 3}$ and  (c) $\Phi(r)=e^{-\alpha r}/\sqrt{r} $ and  $P=0.4 e^{i \pi / 3}$.
}
\label{fig:error}
\end{figure*}
Figure \ref{fig:error} shows the relationship between the error $F(N) $ and $N$ from $2$ to $30$.
We calculate the numerical integration on $\partial \Omega$ by the trapezoidal rule.
The vertical axes of Fig. \ref{fig:error} are shown in the  logarithmic scale.
Since we observe that the error $  F(N)  $ is linearly decreasing against $N$ as in Fig. \ref{fig:error} (a), we see that the error between $g$ and $g_N$ decays exponentially in numerics.
Figure \ref{fig:error} (b) is the numerical result with the same parameters as Fig. \ref{fig:error} (a) except for  the Gauss kernel $\Phi(r)=e^{-r^2}$.
It is observed that the error $ F(N) $ is also linearly decreasing as varying the number of $N$.
Fig. \ref{fig:error}  (c) shows that the numerical result with $\Phi(r)=e^{-\alpha r}/\sqrt{r} $ and  $P=0.4 e^{i \pi / 3}$.
The error is periodically decreasing.
Moreover, it is observed that the error becomes small when $N$ is a multiple of 6.
This seems to be due to that the situation when $N$ is a multiple of  6 has the nearest collocation point on $Re^{\pi/3i}$.

\section{Discussions}\label{sec:summary}
In this paper, we have constructed the approximate solution for the Neumann boundary problem of the modified Helmholtz equation in the disk domain by the MFS.
We have shown the sufficient condition of the existence of the approximate solution by analyzing the eigenvalues of the cyclic matrix associated by the MFS.
In \cite{kats-okamo1988}, the sufficient and necessary condition for the existence and convergence of the approximate solution in the  Dirichlet problems of the Laplace equation has been reported by using the property of the fundamental solution, $\log(|x|)$.
Although our analysis provided the only sufficient, and restricted, condition of the existence of the approximate solution, it may be relaxed.
The necessary condition is left for future work.

Reducing the error into the difference of the function $s$ on $\partial \Omega$ and normal derivative of the approximate solution $g_N$ as the boundary integrations, we established the algorithm to calculate the error bound.
This enabled us to estimate the error between the exact solution $g$ and the approximate solution $g_N$ \textit{a priori}.
We have shown that the error converges to $0$ with $N^2 a^N$ order with $0<a<1$ and number of collocation points $N$.
In \cite{kats-okamo1988}, it has been shown that the error of the exact solution and approximate solution for the Laplace equation of the Dirichlet problem converges to $0$ in exponential order.
The reason for the difference of the degree of $N$ is a result of the set of the Neumann boundary condition, and thus calculating the normal derivative of the modified Bessel function of first kind.
Furthermore, this error estimation enables us to compute the error without constructing the exact solution \textit{a posteriori} in numerical simulations as in Section \ref{sec:simu}.
As demonstrated in Section \ref{sec:simu}, we have shown that the decay order of the error is exponential in the numerics against the number of collocation points $N$.
In addition, we remark that we can introduce the numerical verification for the numerics of the MFS if we compute the error as the boundary integrations with the numerical verification.

As explained in Section \ref{sec:proof2}, the Fourier coefficients, $a_n$, ${\tilde A}_n$, and ${\hat s}_n {\tilde A}_n /f(\omega^n)$ play important roles in the analysis of the error estimate.
Lemmas for $a_n$, $A_n$, and ${\tilde A}_n$ originally come from the property of the modified Bessel functions of the first and second kinds.
If the fundamental solution can be expanded to the Fourier series, the method considered in this paper may be applicable to other problems with another differential operator in the MFS in general.
In addition, as the earlier studies succeeded in extending the results obtained in the disk domain to the general domain in the case of the Laplace equation, we plan to intensify our investigation for the extension thereof in the future by using the results in this paper.

\appendix

\renewcommand{\theequation}{A.\arabic{equation}}
\setcounter{equation}{0}
\section{Exact solution}\label{append:exact}
We will show that the function $ g(r e^{ i \theta } ) $ given in \eqref{eq:exact_g} satisfies equations \eqref{eq:gene_prob}.
Changing the variables in \eqref{eq:gene_prob} as
\[
g(x,y)=g(r\cos \theta, r\sin \theta)=v(r,\theta),\ r= \sqrt{x^2+y^2}, \theta = \arctan \frac  y x,
\]
we have
\begin{align*}
\Delta g &= \frac{ \partial^2 v }{ \partial r^2 } + \frac{1}{r} \frac{\partial v }{ \partial r } + \frac{ 1 }{ r^2 } \frac{\partial^2 v }{ \partial \theta^2 }, \\
\frac{ \partial g }{ \partial n } &= \nabla g \cdot n = v_r.
\end{align*}
Thus, \eqref{eq:gene_prob} is rewritten by
\begin{equation}\label{eq:polar}
\left \{ 
\begin{aligned}
&v_{rr} +\frac{1}{r} v_r  + \frac{1}{r^2} v_{\theta\theta} -\alpha^2  v =0, \quad r\in(0, R] , \theta \in [0, 2 \pi), \\
&v_r (R, \theta) =  S(\theta) , \quad  \theta \in [0, 2 \pi).
\end{aligned}
\right. 
\end{equation}
\begin{proof}[Proof of Proposition \ref{prop:exact}]
From differentiation term by term in \eqref{eq:exact_g}, we compute that 
\begin{align*}
	g_r( r e^{ i \theta } ) = \sum_{n \in Z} a_n e^{i n \theta} = S(\theta),
\end{align*}
which implies the second equation of \eqref{eq:polar}.
For the first equation of \eqref{eq:polar}, we may prove for each $n \in \Z$.
Hence, the function $I_n(\alpha r) e^{i n \theta}$ satisfies the first equation of \eqref{eq:polar}.

Changing the variable $r=\tilde{r}/ \alpha$, 
 this reduces to the fact that the modified Bessel function $I_n( \tilde{r} )$ is a bounded solution at $\tilde{r}=0$ of the differential equation with respect to $ \xi = \xi( \tilde{r} )$:
\[
\tilde{r}^2 \xi_{\tilde{r}\tilde{r}} + \tilde{r} \xi_{\tilde{r}} - (n^2+\tilde{r}^2) \xi = 0.
\]
Thus, the function $I_n( \tilde{r} ) e^{i n \theta} $ satisfies $ \tilde{r}^2 v_{ \tilde{r} \tilde{r}  } + \tilde{r}  v_{\tilde{r} }  + v_{\theta\theta} - \tilde{r} ^2 v =0 $.
\end{proof}

\section{Lagrangian interpolation polynomial}
\label{append:Lag}
In this section we explain the way of the calculation of $G^{-1}$.
To describe the cyclic matrix by a polynomial we introduce the following  $N \times N$ matrix:
\begin{equation*}
J = 
\begin{pmatrix}
0 & 1 & \cdots & 0 \\
\vdots &\ddots&\ddots &\vdots \\
0 & & \ddots & 1\\
1 & 0 & \cdots & 0
\end{pmatrix}.
\end{equation*}
Using this matrix, we can represent the matrix $G$ as 
\begin{equation*}
	G = c_0 I + c_1 J + c_2 J^2 + \cdots +c_{N-1}J^{N-1},
\end{equation*}
where $I$ is the identity matrix.
Supposing $G^{-1} $ is also the cyclic matrix, we set it as 
\begin{equation*}
	G^{-1} = b_0 I + b_1 J + b_2 J^2 + \cdots +b_{N-1}J^{N-1}.
\end{equation*}
For these polynomials, we define the following polynomials as 
\begin{align*}
f(x) &= \sum_{j=0}^{N-1}c_j x^j, \\
f^{-1}(x)& := \sum_{j=0}^{N-1}b_j x^j .
\end{align*}
As $GG^{-1}=I$, we see that $f(x) f^{-1}(x)=1$ by putting $x=J$.
Furthermore, as $G$ and $G^{-1}$ are cyclic, and  $x^n=1$, the polynomial $f(x)$ and $ f^{-1}( x )$ with $x^n=1$  are the eigenvalues of $G$ and $G^{-1}$, respectively.
Then we see that $f(x) f^{-1}(x)-1$ is divisible by $x^n-1$.
This implies that
\[ 
f(\omega^j) f^{-1} (\omega^j)=1, \quad  (j=0, \ldots, N-1).
\]
Assuming $f(\omega^j)\neq0$, we have $f^{-1}(\omega^j)=1/f(\omega^j), \ j=0, \ldots, N-1$.
Using the Lagrangian interpolation polynomial to obtain the coefficients of $g(x)$ yields
\begin{equation*}
	f^{-1}(x) = \sum_{j=0}^{n-1} \frac{ \prod_{k \neq j, 0 \le k \le N-1} (x - \omega^k) }{ \prod_{ k \neq j } ( \omega^j - \omega^k ) } \frac{1}{ f( \omega^j )}.
\end{equation*}
We calculate the numerator and denominator, respectively.
The numerator is computed as
\begin{equation*}
\begin{split}
	\prod_{k \neq j, 0 \le k \le N-1} (x - \omega^k)
	&=\frac{ \ds{\prod_{k = 0 }^{N-1}} (x - \omega^k) }{ x- \omega^j } \\
	&=\frac{x^N-1}{x- \omega^j }\\
	&= \frac{ 1 }{ \omega^j } \frac{ 1 - (\omega^{-j} x)^N }{ 1 - \omega^{-j} x } \\
	&=\frac{ 1 }{ \omega^j } \sum_{k = 0}^{N-1} (\omega^{-j} x) ^ k.
\end{split}
\end{equation*}
The denominator is calculated as
\begin{equation*}
\begin{split}
	\prod_{ k \neq j } ( \omega^j - \omega^k )
	&=\lim_{x \to \omega^j} \frac{x^N-1}{x - \omega^j} \\
	&=\lim_{x \to \omega^j} Nx^{N-1} = N \omega^{-j}.
\end{split}
\end{equation*}
Therefore, we see that 
\begin{equation*}
	f^{-1}(x)= \frac{1}{N} \sum_{j=0}^{ N-1}  \sum_{ k = 0}^{N-1}  \frac{ (\omega^{-j} x)^k }{f( \omega^j )}
	      = \frac{1}{N} \sum_{ k = 0}^{N-1} \left( \sum_{j=0}^{ N-1}  \frac{ 1 }{\omega^{- jk }f(\omega^ j )} \right) x^k
\end{equation*}
Finally we obtain the exact form of the components of $G^{-1}$ as
\[
b_k=  \frac{1}{N} \sum_{ j = 0}^{N-1} \frac{ 1 }{\omega^{- j k }f(\omega ^ j )}, \quad (k =0, \ldots, N-1).
\]

\renewcommand{\theequation}{C.\arabic{equation}}
\setcounter{equation}{0}
\section{Proof of Lemma \ref{lemm:trace_cir} }
\label{appen:lemm_trace}
In this section, we explain the proof of the Lemma \ref{lemm:trace_cir}.
We suppose $\Omega:=B(0, R) \subset \R^2 $.
We define the functional space $ C^{m,p}( {\Omega} )$ as follows:
\[
	C^{m,p}( {\Omega} )=\{ u \in C^p(\Omega) | \ D^{\alpha}u \in L^p(\Omega), \ (|\alpha| \le m)  \},
\]
where $|\alpha|=\alpha_1+\cdots+\alpha_n$ is the multi index, and $D^{\alpha}u=\partial^{|\alpha|}u/(\partial x_1^{\alpha_1}\cdots \partial x_n^{\alpha_n})$.
As a preparation, we first show the inequality \eqref{ineq:trace_cir} for $u \in C^{1,2}( {\bar \Omega} )$, and thereafter, we prove 
the inequality \eqref{ineq:trace_cir} for $u \in H^1(\Omega)$ by using the density.
We calculate that 
\begin{equation*}
\begin{split}
	u^2(R, \theta)
	&= \frac{1}{R^2} \int_{0}^{R} \frac{\partial }{ \partial r}( r^2 u^2 (r, \theta) ) dr  \\
	&= \frac{2}{R^2} \int_{0}^{R} ( r^2 u u_r+r u^2 ) (r, \theta) dr \\
	\end{split}
\end{equation*}
Since $\nabla u \cdot (x_1,x_2) = r u_r(\sin^2 \theta + \cos^2 \theta) = r u_r$, and from the Schwarz inequality, we have  
\begin{equation*}
\begin{split}
	u^2(R, \theta)
	&= \frac{2}{R^2} \int _{0}^{R}  \left( r u \nabla u \cdot (x_1,x_2) + ru^2 \right) (r, \theta) dr \\
	&\le \frac{2}{R^2} \int _{0}^{R} r \left( R |u | | \nabla u | + u^2 \right) (r, \theta)  dr.
\end{split}
\end{equation*}
Thus we see that 
\begin{equation}\label{ine:prepa}
\begin{split}
	\left\| u \right \|_{L^2(\partial \Omega )}^2
	&=  R \int_{0}^{2 \pi} u(R,\theta)^2d\theta\\
	&\le \frac{2}{R}\int_{\Omega} (R | u | | \nabla u | + u^2 ) dx_1 dx_2\\
	&\le \left( \left\|u \right\|^2_{L^2(\Omega)} + \left\| \nabla u\right\|^2_{L^2(\Omega)}  \right) 
	       +\frac{2}{R} \left\|u \right\|^2_{L^2(\Omega)} \\
	&\le \left( 1 + \frac{2}{R} \right)\left\| u \right\|^2_{H^1(\Omega)}.
\end{split}
\end{equation}

\begin{proof}[Proof of Lemma \ref{lemm:trace_cir}]
We show the inequality \eqref{ineq:trace_cir} for $u \in H^1(\Omega)$ by using the density.
As $C^{1,2}(\Omega)$ is dense in $H^{1}(\Omega)$, for any $u \in H^1(\Omega)$, there exists a sequence $\{ u_n \}_{n\in N} \subset C^{1,2}(\Omega)$ such that $ \| u - u_n \|_{H^1(\Omega)} \to 0, \ (n \to \infty)$.
Using the inequality \eqref{ine:prepa}, we see that 
\begin{equation*}
\begin{split}
	\| u_n - u_m \|_{L^2(\partial \Omega)}
	&\le \sqrt{  1 + \frac{2}{R} } \left\| u_n - u_m \right\|_{H^1(\Omega)}\\
	&\le \sqrt{ 1 + \frac{2}{R} } \left( \left\| u_n - u \right\|_{H^1(\Omega)}+ \left\| u - u_m \right\|_{H^1(\Omega)} \right) \\
	&\to 0.
\end{split}
\end{equation*}
This implies that $\{ u_n \}_{n\in N} \subset C^{1,2}(\Omega)$ is the Cauchy sequence in $L^2(\partial \Omega)$.
Then, there exists a function $v \in L^2(\partial \Omega)$ such that $ \left\| u_n - v \right\|^2_{ L^2(\partial \Omega) } \to 0, \ (n \to \infty)$.
Thus, for any $u \in H^1(\Omega)$, a function $v$ in $L^2(\partial \Omega)$ is determined.
We define the trace of $u\in H^1(\Omega)$ in $\partial \Omega$ as 
\[
u|_{\partial \Omega}:=v.
\]
Next we show that the determination of $u|_{\partial \Omega}$ does not depend on the choice of the sequences in $C^{1,2}(\Omega)$.
We consider other sequence in $C^{1,2}(\Omega)$ that converges to $u \in H^1(\Omega)$.
Suppose that for any $u \in H^1(\Omega)$, there exists another sequence $\{ v_n\}$ in $C^{1,2}(\Omega)$ such that $ \| u - v_n \|_{H^1(\Omega)} \to 0, \ (n \to \infty)$.
Then we compute that 
\begin{equation*}
\begin{split}
	\| v - v_n \|_{L^2(\partial \Omega)}
	&\le \| v - u_n \|_{L^2(\partial \Omega)} + \| u_n - v_n \|_{L^2(\partial \Omega)}\\
	&\le \| v - u_n \|_{L^2(\partial \Omega)} + \sqrt{ 1 + \frac{2}{R} } \left( \left\| u_n - u \right\|_{H^1(\Omega)}+ \left\| u - v_n \right\|_{H^1(\Omega)} \right) \\
	&\to 0, \quad (n\to 0).
\end{split}
\end{equation*}
Then $u|_{\partial \Omega}$ is well defined.
Finally, we show the inequality  \eqref{ineq:trace_cir} as 
\begin{equation*}
\begin{split}
	\| u\|_{L^2(\partial \Omega)} = \| v\|_{ L^2(\partial \Omega) }
	&= \lim_{n\to \infty} \| u_n\|_{L^2(\partial \Omega)}\\
	&\le  \lim_{n\to \infty}  \sqrt{ 1 + \frac{2}{R} }\left\| u_n \right\|^2_{H^1(\Omega)}\\
	&= \sqrt{ 1 + \frac{2}{R} } \left\| u \right\|^2_{H^1(\Omega)}.
\end{split}
\end{equation*}
\end{proof}

\section*{Acknowledgments}
The authors would like to thank Professor Mitsuhiro Nakao of Waseda University for the fruitful suggestions for the error estimation in Section \ref{sec:NV}.
This work was supported in part by JST CREST Grant Number  JPMJCR14D3 to S.-I. E.,  JSPS KAKENHI Grant Number  15H03613 to H. O., and JSPS KAKENHI Grant Numbers  17K14228 and 20K14364 to Y. T..

\bibliographystyle{siamplain}
\bibliography{references}

\end{document}